\documentclass[10pt,reqno,a4paper]{amsart}
\usepackage{amsthm}
\usepackage{mathtools}
\usepackage{amssymb}
\usepackage{hyperref}
\usepackage{enumitem}
\usepackage{graphicx}
\usepackage{xcolor}
\usepackage{kotex}
\usepackage{comment}

\numberwithin{equation}{section}
\numberwithin{figure}{section}
\allowdisplaybreaks

\theoremstyle{plain}
\newtheorem{thm}{\protect\theoremname}[section]
\newtheorem{lem}[thm]{\protect\lemmaname}
\newtheorem{prop}[thm]{\protect\propositionname}

\theoremstyle{remark}
\newtheorem{rem}[thm]{\protect\remarkname}

\providecommand{\theoremname}{Theorem}
\providecommand{\corollaryname}{Corollary}
\providecommand{\lemmaname}{Lemma}
\providecommand{\remarkname}{Remark}
\providecommand{\propositionname}{Proposition}

\begin{document}
\global\long\def\Im{\mathrm{Im}}%
\global\long\def\Re{\mathrm{Re}}%
\global\long\def\Hc{\mathcal{H}}%
\global\long\def\M{\mathbb{M}}%
\global\long\def\P{\mathbb{P}}%
\global\long\def\L{\mathcal{L}}%
\global\long\def\F{\mathcal{F}}%
\global\long\def\s{\sigma}%
\global\long\def\G{\mathcal{G}}%
\global\long\def\d{\partial}%
\global\long\def\mc#1{\mathcal{#1}}%
\global\long\def\Right{\Rightarrow}%
\global\long\def\Left{\Leftarrow}%
\global\long\def\les{\lesssim}%
\global\long\def\hook{\hookrightarrow}%
\global\long\def\D{\mathbf{D}}%
\global\long\def\rad{\mathrm{rad}}%
\global\long\def\jp#1{\langle#1\rangle}%
\global\long\def\norm#1{\|#1\|}%
\global\long\def\ol#1{\overline{#1}}%
\global\long\def\wt#1{\widehat{#1}}%
\global\long\def\tilde#1{\widetilde{#1}}%
\global\long\def\br#1{(#1)}%
\global\long\def\Bb#1{\Big(#1\Big)}%
\global\long\def\bb#1{\big(#1\big)}%
\global\long\def\lr#1{\left(#1\right)}%
\global\long\def\la{\lambda}%
\global\long\def\al{\alpha}%
\global\long\def\be{\beta}%
\global\long\def\ga{\gamma}%
\global\long\def\La{\Lambda}%
\global\long\def\De{\Delta}%
\global\long\def\na{\nabla}%
\global\long\def\fl{\flat}%
\global\long\def\sh{\sharp}%
\global\long\def\calN{\mathcal{N}}%
\global\long\def\avg{\mathrm{avg}}%
\global\long\def\bbR{\mathbf{\mathbb{R}}}%
\global\long\def\bbC{\mathbf{\mathbb{C}}}%
\global\long\def\bbZ{\mathbf{\mathbb{Z}}}%
\global\long\def\bbN{\mathbf{\mathbb{N}}}%
\global\long\def\bbT{\mathbb{T}}%
\global\long\def\bfD{\mathbf{D}}%
\global\long\def\bfL{\mathbf{L}}%
\global\long\def\calF{\mathcal{F}}%
\global\long\def\calH{\mathcal{H}}%
\global\long\def\calL{\mathcal{L}}%
\global\long\def\calE{\mathcal{E}}%
\global\long\def\calO{\mathcal{O}}%
\global\long\def\calR{\mathcal{R}}%
\global\long\def\calD{\mathcal{D}}%
\global\long\def\calI{\mathcal{I}}%
\global\long\def\calT{\mathcal{T}}%
\global\long\def\Lmb{\Lambda}%
\global\long\def\eps{\varepsilon}%
\global\long\def\lmb{\lambda}%
\global\long\def\gmm{\gamma}%
\global\long\def\rd{\partial}%
\global\long\def\chf{\mathbf{1}}%
\global\long\def\td#1{\widetilde{#1}}%
\global\long\def\sgn{\mathrm{sgn}}%
\global\long\def\blambda{\boldsymbol\lambda}%
\global\long\def\biota{\boldsymbol\iota}%
\global\long\def\Red#1{\textcolor{red}{#1}}%
\global\long\def\NL{\mathrm{NL}}%

\title[No bubble tree]{No bubble trees for the $1$-equivariant harmonic map heat flow and the radial energy-critical nonlinear heat equation in low dimensions}


\author{Taegyu Kim}
\email{k1216300@kias.re.kr}
\address{School of Mathematics, Korea Institute for Advanced Study, 85 Hoegiro Dongdaemun-gu, Seoul 02455, Korea}

\begin{abstract}
	We consider the $1$-equivariant harmonic map heat flow (HMHF) from $\mathbb R^2$ to $\mathbb S^2$ and the radial energy-critical nonlinear heat equation (NLH) in dimensions $d=3,4,5$. We prove that every finite-energy solution of (HMHF) and every $\dot H^1$-bounded solution of (NLH) has at most one bubble: every finite-time blow-up has exactly one bubble, whereas every global solution has either no bubble or one bubble at infinite time. Starting from the soliton resolution, we exclude bubble trees by modulation analysis, deriving a contradiction from the relative dynamics of the two innermost scales. This energy method does not rely on maximum principle and, in particular, applies without restriction on the bubble signs.
\end{abstract}	
\maketitle

\tableofcontents

\section{Introduction}

We consider the energy-critical $D$-equivariant harmonic map heat flow
\begin{equation}\tag{HMHF}\label{eq:HMHF}
	 v_t=v_{rr}+\frac1r v_r-\frac{D^2\sin(2v)}{2r^2}, \qquad (t,r)\in[0,T_+)\times(0,\infty)
\end{equation}
with equivariance index $D\in\mathbb Z_{\geq1}$ and the radial focusing energy-critical nonlinear heat equation
\begin{equation}\tag{NLH}\label{eq:NLH}
	 u_t=u_{rr}+\frac{d-1}{r}u_r+|u|^{\frac4{d-2}}u, \qquad (t,r)\in[0,T_+)\times(0,\infty)
\end{equation}
in dimensions $d\geq3$. 

Our main results, Theorems~\ref{thm:HMHF main} and \ref{thm:NLH main}, rule out bubble trees for finite-energy solutions to \eqref{eq:HMHF} when $D=1$ and for $\dot H^1_d$-bounded radial solutions to \eqref{eq:NLH} when $d\in\{3,4,5\}$, in both the finite-time blow-up and global cases. In particular, these results sharpen the soliton resolution to a \emph{one-bubble resolution}: every finite-time blow-up has exactly one bubble, while every global solution has at most one bubble at infinite time.

\subsection{The equivariant harmonic map heat flow}

The harmonic map heat flow for maps $\Phi:[0,T_+)\times\bbR^2\to\mathbb S^2$ is the $L^2$-gradient flow of the Dirichlet energy
\begin{equation*}
	E(\Phi)\coloneqq\frac12\int_{\bbR^2}|\nabla\Phi|^2dx.
\end{equation*}
In extrinsic form, the equation is
\begin{equation*}
	\partial_t\Phi=\Delta\Phi+|\nabla\Phi|^2\Phi.
\end{equation*}
The harmonic map heat flow for maps between Riemannian manifolds was introduced by Eells and Sampson \cite{EellsSampson1964AJM} to study deformations of maps toward harmonic maps, the critical points of the Dirichlet energy. For $\lambda>0$, the flow is invariant under the scaling $\Phi(t,x)\mapsto\Phi(t/\lambda^2,x/\lambda)$, which also preserves the Dirichlet energy. Thus the two-dimensional problem is energy-critical.

The energy-class Cauchy theory was established by Struwe, originally for maps from a closed surface \cite{Struwe1985}. The associated global weak solution is smooth away from finitely many spacetime points, and suitable rescalings near each singular point converge locally to a nonconstant harmonic map. The energy identity underlying this bubbling mechanism and the associated lack of compactness for Palais--Smale sequences were studied further by Qing \cite{Qing1995}, Ding--Tian \cite{DingTian1995}, Wang \cite{Wang1996}, Qing--Tian \cite{QingTian1997CPAM}, and Lin--Wang \cite{LinWang1998CVPDE}. Along suitable sequences approaching a singular time, the solution decomposes into rescaled harmonic maps and a body map. This sequential bubbling theory was developed further by Topping \cite{Topping1997JDG,Topping2004AnnMath,Topping2004MathZ}.

We restrict our attention to equivariant solutions, for which the flow reduces to a scalar radial equation. Let $D\in\mathbb Z_{\geq1}$. In polar coordinates $(r,\theta)$ on $\bbR^2$, the $D$-equivariant symmetry class consists of maps of the form
\begin{equation*}
	\Phi(t,r,\theta)=\big(\sin v(t,r)\cos(D\theta),\sin v(t,r)\sin(D\theta),\cos v(t,r)\big).
\end{equation*}
This class is preserved by the flow, and the equation for $v$ is \eqref{eq:HMHF}. The Dirichlet energy reduces to
\begin{equation*}
	E(v)\coloneqq\pi\int_0^\infty\left(v_r^2+D^2\frac{\sin^2v}{r^2}\right)r dr.
\end{equation*}
The finite-energy space is the disjoint union of the components
\begin{equation*}
	\calE_{\ell,m}\coloneqq\left\{v:E(v)<\infty,\ \lim_{r\to0}v(r)=\ell\pi,\ \lim_{r\to\infty}v(r)=m\pi\right\}, \qquad \ell,m\in\bbZ.
\end{equation*}
The difference of two maps in the same component belongs to $\calE_{0,0}$. We equip $\calE_{0,0}$ with the norm
\begin{equation*}
	\|h\|_{\calE}^2\coloneqq\int_0^\infty\left(h_r^2+\frac{h^2}{r^2}\right)r dr,
\end{equation*}
which induces a metric on every $\calE_{\ell,m}$. Moreover, for a maximal solution $v$ on $[0,T_+)$ with $T_+\in(0,\infty]$, we have
\begin{equation*}
	E(v(t_2))+2\pi\int_{t_1}^{t_2}\|v_t(t)\|_{L^2(\bbR^2)}^2dt=E(v(t_1)),\qquad 0\leq t_1\leq t_2<T_+.
\end{equation*}
The normalized nonconstant stationary solution of \eqref{eq:HMHF} is the $D$-equivariant harmonic map $Q\in\calE_{0,1}$ given by
\begin{equation*}
	Q(r)=2\arctan(r^D).
\end{equation*}

For finite-energy equivariant solutions, Jendrej and Lawrie strengthened the sequential description above to a continuous-in-time bubble decomposition \cite{JL2023CVPDE}. Without symmetry, continuous-in-time convergence to the family of multi-bubble configurations was established by Jendrej, Lawrie, and Schlag \cite{JLS2025Pi}. Proposition~\ref{prop:sol resol hmhf} records the $1$-equivariant decomposition used below. However, this equivariant decomposition leaves open how many bubbles can occur.

Van der Hout proved, using the maximum principle, that finite-time blowing up bubble trees do not occur for equivariant harmonic map heat flows from the disk into $\mathbb S^2$ \cite{Hout2003JDE}. Here we introduce a new approach that precludes bubble trees both at finite-time blow-up and in the global case. To state our main result, write $Q_{[\lambda]}(r)=Q(r/\lambda)$.

\begin{thm}[One-bubble resolution for $1$-equivariant \eqref{eq:HMHF}]\label{thm:HMHF main}
	Let $D=1$, $\ell,m\in\mathbb Z$, and $v_0\in\calE_{\ell,m}$ be an initial datum; let $v(t)$ be the corresponding solution to \eqref{eq:HMHF} with $T_+\in(0,\infty]$ its maximal time of existence.

	\emph{(Finite-time blow-up case)} If $T_+<\infty$, then there exist a sign $\iota \in\{\pm1\}$, a continuous function $\lambda(t)\in(0,\infty)$, and a function $v^{\ast}\in\calE_{0,m-\ell-\iota}$ such that
	\begin{equation}\label{eq:hmhf decom}
		\left\| v(t)-\ell\pi-\iota Q_{[\lambda(t)]}-v^{\ast}\right\|_{\calE}
		+\frac{\lambda(t)}{\sqrt{T_+-t}}\to0\quad\text{as}\quad t\to T_+.
	\end{equation}

	\emph{(Global case)} If $T_+=\infty$, then either
	\begin{equation*}
		\ell=m\quad \text{and} \quad \left\|v(t)-\ell\pi\right\|_{\calE}\to0\quad\text{as}\quad t\to\infty,
	\end{equation*}
	or there exist a sign $\iota\in\{\pm1\}$ and a continuous function $\lambda(t)\in(0,\infty)$ such that
	\begin{equation}\label{eq:hmhf convergence}
		\ell+\iota=m\quad \text{and} \quad
		\left\|v(t)-\ell\pi-\iota Q_{[\lambda(t)]}\right\|_{\calE}
		+\frac{\lambda(t)}{\sqrt{t}}\to0\quad\text{as}\quad t\to\infty.
	\end{equation}
	
	In particular, if $|\ell-m|\geq2$, then $T_+<\infty$.
\end{thm}

Thus, when $D=1$, the decomposition in Proposition~\ref{prop:sol resol hmhf} contains exactly one bubble when $T_+<\infty$ and at most one bubble when $T_+=\infty$. In particular, bubble trees do not occur.

Both one-bubble scenarios in the theorem are known to occur. Finite-time blow-up in the $1$-equivariant setting was first constructed by Chang, Ding, and Ye \cite{CDY1992JDG}. Rapha\"el and Schweyer later constructed finite-time one-bubble blow-up regimes and gave sharp descriptions of their dynamics \cite{RaphaelSchweyer2013CPAMHeat,RaphaelSchweyer2014AnalPDEHeatQuantized}. Global one-bubble regimes were constructed by Wei, Zhang, and Zhou \cite{WZZ2026JFA}. Our theorem does not, however, determine the asymptotic law of the remaining scale.

\subsection{The radial energy-critical nonlinear heat equation}
Our analysis also applies to the energy-critical nonlinear heat equation.
For $d\geq3$, we consider the energy-critical nonlinear heat equation on $\bbR^d$,
\begin{equation*}
	\partial_tu=\Delta_{\bbR^d} u+|u|^{\frac4{d-2}}u.
\end{equation*}
This equation is the $L^2$-gradient flow associated with the energy
\begin{equation*}
	E(u)\coloneqq\int_{\bbR^d}\left(\frac12|\nabla u|^2-\frac{d-2}{2d}|u|^{\frac{2d}{d-2}}\right)dx.
\end{equation*}
For $\lambda>0$, the scaling $u(t,x)\mapsto\lambda^{-(d-2)/2}u(\lambda^{-2}t,\lambda^{-1}x)$ preserves both the equation and the energy. Thus $\dot H^1(\bbR^d)$ is the critical energy space. The Cauchy problem is well-posed in this space \cite{Weissler1980,BrezisCazenave1996}. Moreover, if $u$ is a maximal solution on $[0,T_+)$, then for $0\leq t_1\leq t_2<T_+$,
\begin{equation*}
	E(u(t_2))+\int_{t_1}^{t_2}\|u_t(t)\|_{L^2(\bbR^d)}^2dt=E(u(t_1)).
\end{equation*}
Up to scaling and sign, the unique nonzero radial finite-energy stationary solution is the Aubin--Talenti bubble
\begin{equation*}
	W(r)=\left(1+\frac{r^2}{d(d-2)}\right)^{-(d-2)/2},
\end{equation*}
which satisfies $-\Delta_{\bbR^d} W=W^{(d+2)/(d-2)}$.

We now restrict our attention to $\dot H^1(\bbR^d)$-bounded radial solutions. Writing $u(t,x)=u(t,r)$ with $r=|x|$ reduces the equation to \eqref{eq:NLH}. For these solutions, Aryan proved continuous-in-time soliton resolution in every dimension $d\geq3$ \cite{Aryan2024solResolHeat}. Proposition~\ref{prop:sol resol nlh} records the precise statement used below. As in the \eqref{eq:HMHF} setting, the soliton resolution does not determine how many bubbles can occur.

Our second main result rules out bubble trees in low dimensions $d\in\{3,4,5\}$, both at finite time and globally. For $\lambda>0$, write $W_\lambda(r)=\lambda^{-(d-2)/2}W(r/\lambda)$.

\begin{thm}[One-bubble resolution for radial \eqref{eq:NLH} in low dimensions]\label{thm:NLH main}
	Let $d\in\{3,4,5\}$, and let $u\in C([0,T_+);\dot H^1(\bbR^d))$ be a radial solution of \eqref{eq:NLH} with $T_+\in(0,\infty]$ its maximal time of existence. Assume 
	\begin{equation*}
		\sup_{0\leq t<T_+}\|u(t)\|_{\dot H^1(\bbR^d)}<\infty.
	\end{equation*}

	\emph{(Finite-time blow-up case)} If $T_+<\infty$, then there exist a sign $\iota\in\{\pm1\}$, a continuous function $\lambda(t)\in(0,\infty)$, and a function $u^{\ast}\in\dot H^1(\bbR^d)$ such that
	\begin{equation}\label{eq:NLH decom}
		\left\|u(t)-\iota W_{\lambda(t)}-u^{\ast}\right\|_{\dot H^1(\bbR^d)}
		+\frac{\lambda(t)}{\sqrt{T_+-t}}\to0\quad\text{as}\quad t\to T_+.
	\end{equation}

	\emph{(Global case)} If $T_+=\infty$, then either
	\begin{equation*}
		\|u(t)\|_{\dot H^1(\bbR^d)}\to0\quad\text{as}\quad t\to\infty,
	\end{equation*}
	or there exist a sign $\iota\in\{\pm1\}$ and a continuous function $\lambda(t)\in(0,\infty)$ such that
	\begin{equation}\label{eq:NLH convergence}
		\left\|u(t)-\iota W_{\lambda(t)}\right\|_{\dot H^1(\bbR^d)}
		+\frac{\lambda(t)}{\sqrt{t}}\to0\quad\text{as}\quad t\to\infty.
	\end{equation}
\end{thm}

Consequently, for $d\in\{3,4,5\}$, Proposition~\ref{prop:sol resol nlh} reduces to one bubble at finite time and to the zero- or one-bubble alternative globally. The asymptotic law of the bubble scale remains open.

One-bubble solutions have been constructed in both the finite-time blow-up and global cases for $d\in\{3,4,5\}$. For finite-time blow-up cases, see \cite{Schweyer2012JFA,PMW2019,delPinoMussoWeiZhang2020HeatQuantized,Harada2020AIHPC}; for the global solutions, see \cite{delPinoMussoWei2020Infheat,WZZ2024JDE,LWZZ2024}.

For the discussion and the common analysis below, we follow the unified notation in Section~\ref{sec:notation} and set 
\begin{equation*}
	d=2D+2,
\end{equation*}
using $D$ as the common parameter for the two equations. This convention is motivated by the identity
\begin{equation*}
	\left(\partial_{rr}+\frac1r\partial_r-\frac{D^2}{r^2}\right)(r^Du)
	=r^D\left(\partial_{rr}+\frac{d-1}{r}\partial_r\right)u.
\end{equation*}

\subsection{Discussions on the results}

\begin{rem}[Method and novelties]
	Our proof is based on the energy method with modulation analysis, following the one-bubble argument of \cite{CollotMerleRaphael2017CMP} and its multi-bubble refinement in \cite{KimMerle2025CPAM,KimMerle2026arXiv} in high dimensions. It does not use the maximum principle and applies without any restriction on the signs of bubbles. The key point in the present work is to capture the dynamics of the two innermost scales, $\lambda_N$ relative to $\lambda_{N-1}$, rather than the evolution of $\lambda_N$ alone. For \eqref{eq:HMHF}, our argument not only recovers, by a different method, van der Hout's exclusion of finite-time blowing up bubble trees for equivariant flows from the disk \cite{Hout2003JDE}, but also extends this exclusion to the global case.
\end{rem}

\begin{rem}[Equivariance and spatial dimension]
	Under the correspondence $d=2D+2$, $D=2$ is expected to mark the threshold at which the global dynamics change. For $D>2$, K.~Kim and Merle prove that finite-time blow-up does not occur and completely classify the global dynamics for finite-energy solutions to \eqref{eq:HMHF} and $\dot H^1(\bbR^d)$-bounded radial solutions to \eqref{eq:NLH}, including the number and signs of the bubbles and the asymptotic laws of their scales \cite{KimMerle2025CPAM}. A related result for \eqref{eq:NLH} without symmetry appears in \cite{KimMerle2026arXiv}.

    Heuristically, the distinction between the two regimes reflects the slow decay of the bubble tails in low dimensions: when $0<D<2$, the interaction between neighboring bubbles remains too strong under scale separation for a bubble tree to persist.
    
	The borderline case $D=2$ is not covered here; from the viewpoint of \cite{KimMerle2025CPAM}, we expect all finite-energy \eqref{eq:HMHF} solutions and all $\dot H^1$-bounded radial \eqref{eq:NLH} solutions to be global and both equations to admit infinite-time bubble trees with an arbitrary number of bubbles. The $\dot H^1$-boundedness assumption for \eqref{eq:NLH} is essential: at $D=2$, or equivalently $d=6$, Harada constructed a radial type II finite-time blow-up solution that is not $\dot H^1(\bbR^6)$-bounded \cite{Harada2020annPDE}.
\end{rem}

\begin{rem}[On bubble trees]
	Bubble trees (or bubble towers) have been constructed in several parabolic settings. Infinite-time or ancient trees are obtained for \eqref{eq:HMHF} with a suitable target, the Yamabe flow, and the high-dimensional energy-critical heat equation in \cite{Topping2000,DPS2018crelle,delPinoMussoWei2021AnalPDE,SunWeiZhang2021arXiv}. For the classification of high-dimensional energy-critical heat flows, see \cite{KimMerle2025CPAM,KimMerle2026arXiv}. In the disk problem for \eqref{eq:HMHF}, van der Hout excludes finite-time bubble trees when $D=1$ \cite{Hout2003JDE}. This conclusion was recently extended to every $D\geq1$ in \cite{Samuelian2026CVPDE}; see also \cite{Samuelian2026arXiv}.
	
	In the dispersive setting, concentric two-bubble solutions are constructed in \cite{Jendrej2019AJM,Jendrej2017AnalPDE,JendrejLiXu2026arXiv,JendrejLiXu2026arXivNLS}. For wave maps, the infinite-time two-bubble dynamics at threshold energy is classified in \cite{JendrejLawrie2018Invent}, while finite-time trees with two and then arbitrarily many bubbles are constructed in \cite{JendrejKrieger2025arXiv,KriegerPalacios2026arXiv}, and infinite-time trees with arbitrarily many bubbles in \cite{HwangKim2026arXiv}. By contrast, at most one soliton occurs for the equivariant self-dual Chern--Simons--Schr\"odinger equation \cite{KimKwonOh2025AJM}, and bubble trees are excluded for the Calogero--Moser derivative nonlinear Schr\"odinger equation in \cite{KimTKwon2024arxivSolResol}. See also \cite{Shen2026arXiv} for the radial three-dimensional energy-critical nonlinear wave equation.
\end{rem}

\begin{rem}[One-bubble classification]
	In the range $0<D<2$ considered here, our theorems show that every solution contains at most one bubble. The remaining problem is to classify its possible scale laws and their dependence on the radiation. For critical \eqref{eq:NLH}, the formal analysis of Filippas, Herrero, and Vel\'azquez predicts discrete finite-time rates \cite{FilippasHerreroVelazquez2000}. Both finite- and infinite-time one-bubble regimes are known to occur for \eqref{eq:HMHF} and \eqref{eq:NLH}, but the existing constructions do not provide an exhaustive classification.
	
	For $D\geq2$, near-one-bubble dynamics is classified for \eqref{eq:HMHF} at $D=2$ in \cite{GustafsonNakanishiTsai2010CMP}. For \eqref{eq:NLH}, the corresponding dynamics is classified when $D>2$ in \cite{CollotMerleRaphael2017CMP} and at $D=2$ in \cite{Harada2026CVPDE}. Related classifications in critical parabolic models appear, for example, in \cite{Mizoguchi2022CPAM}. For analogous results in critical dispersive models, see \cite{Raphael2005MathAnnalen, MerleRaphael2006JAMS,MartelMerleRaphael2014Acta,Kim2025JEMS,JeongKimKimKwon2026arXiv}.
\end{rem}

\begin{rem}[Related results for semilinear heat equations]
	The finite-time dynamics of semilinear heat equations depend strongly on the exponent, the sign of the solution, and the imposed symmetry. Here type I blow-up refers to the ODE rate, while type II blow-up is faster. In the subcritical range, every finite-time blow-up is of type I by Giga--Kohn \cite{GigaKohn1985CPAM} and Giga--Matsui--Sasayama \cite{GMS2004}. When $p>p_{JL}$, where $p_{JL}$ denotes the Joseph--Lundgren exponent, Herrero and Vel\'azquez constructed radial type II solutions with discrete polynomial rates \cite{HerreroVelazquez1992,HerreroVelazquez1994CRASPSI}, and Mizoguchi classified the rates of radial nonnegative type II solutions under additional assumptions \cite{Mizoguchi2007Math.Ann.,Mizoguchi2011TranAMS}. Between the Sobolev critical and Joseph--Lundgren exponents, Matano and Merle excluded radial type II blow-up \cite{MatanoMerle2004CPAM}. At the energy-critical exponent in dimensions $d\geq 7$, Wang and Wei obtained the corresponding exclusion for nonnegative data without radial symmetry \cite{WangWei2021arXiv}.
\end{rem}

\subsection{Statements of soliton resolution}

We record the precise bubble decomposition for finite-energy $1$-equivariant \eqref{eq:HMHF} and soliton resolution for $\dot H^1(\bbR^d)$-bounded radial \eqref{eq:NLH}, $d\geq3$, used below.

\begin{prop}[Bubble decomposition for $1$-equivariant \eqref{eq:HMHF}, \cite{JL2023CVPDE}]\label{prop:sol resol hmhf}
	Let $\ell,m\in\mathbb Z$ and $v_0\in\calE_{\ell,m}$ be an initial datum; let $v(t)$ be the corresponding solution to \eqref{eq:HMHF} with $T_+\in(0,\infty]$ its maximal time of existence.
	
	\emph{(Finite-time blow-up case)} If $T_+<\infty$, then there exist an integer $N\geq1$, signs $\iota_1,\dots,\iota_N\in\{\pm1\}$, continuous functions $\wt\lambda_1(t),\dots,\wt\lambda_N(t)\in(0,\infty)$, and a function $v^{\ast}\in\calE_{0,m^{\ast}}$ with $m^{\ast}=m-\ell-\sum_{j=1}^N\iota_j$ such that
	\begin{equation*}
		\Big\| v(t)-\ell\pi-\sum_{j=1}^N\iota_jQ_{[\wt\lambda_j(t)]}-v^{\ast}\Big\|_{\calE}
		+\sum_{j=2}^N\frac{\wt\lambda_j(t)}{\wt\lambda_{j-1}(t)}
		+\frac{\wt\lambda_1(t)}{\sqrt{T_+-t}}\to0\quad\text{as}\quad t\to T_+.
	\end{equation*}
	
	\emph{(Global case)} If $T_+=\infty$, then there exist an integer $N\geq0$, signs $\iota_1,\dots,\iota_N\in\{\pm1\}$, and continuous functions $\wt\lambda_1(t),\dots,\wt\lambda_N(t)\in(0,\infty)$ such that
	\begin{equation*}
		\Big\| v(t)-\ell\pi-\sum_{j=1}^N\iota_jQ_{[\wt\lambda_j(t)]}\Big\|_{\calE}
		+\sum_{j=2}^N\frac{\wt\lambda_j(t)}{\wt\lambda_{j-1}(t)}
		+\frac{\wt\lambda_1(t)}{\sqrt{t}}\to0\quad\text{as}\quad t\to T_+.
	\end{equation*}
\end{prop}

The corresponding statement for radial \eqref{eq:NLH} is as follows.

\begin{prop}[Soliton resolution for radial \eqref{eq:NLH}, \cite{Aryan2024solResolHeat}]\label{prop:sol resol nlh}
	Let $d\geq3$, and let $u\in C([0,T_+);\dot H^1(\bbR^d))$ be a radial solution of \eqref{eq:NLH} with $T_+\in(0,\infty]$ its maximal time of existence. Assume $\sup_{0\leq t<T_+}\|u(t)\|_{\dot H^1(\bbR^d)}<\infty$.
	
	\emph{(Finite-time blow-up case)} If $T_+<\infty$, then there exist an integer $N\geq1$, signs $\iota_1,\dots,\iota_N\in\{\pm1\}$, continuous functions $\wt\lambda_1(t),\dots,\wt\lambda_N(t)\in(0,\infty)$, and a function $u^{\ast}\in \dot H^1(\bbR^d)$ such that
	\begin{equation*}
		\Big\|u(t)-\sum_{j=1}^N\iota_jW_{\wt\lambda_j(t)}-u^{\ast}\Big\|_{\dot H^1(\bbR^d)}
		+\sum_{j=2}^N\frac{\wt\lambda_j(t)}{\wt\lambda_{j-1}(t)}
		+\frac{\wt\lambda_1(t)}{\sqrt{T_+-t}}\to0\quad\text{as}\quad t\to T_+.
	\end{equation*}
	
	\emph{(Global case)} If $T_+=\infty$, then there exist an integer $N\geq0$, signs $\iota_1,\dots,\iota_N\in\{\pm1\}$, and continuous functions $\wt\lambda_1(t),\dots,\wt\lambda_N(t)\in(0,\infty)$ such that
	\begin{equation*}
		\Big\|u(t)-\sum_{j=1}^N\iota_jW_{\wt\lambda_j(t)}\Big\|_{\dot H^1_d}
		+\sum_{j=2}^N\frac{\wt\lambda_j(t)}{\wt\lambda_{j-1}(t)}
		+\frac{\wt\lambda_1(t)}{\sqrt{t}}\to0\quad\text{as}\quad t\to T_+.
	\end{equation*}
\end{prop}
We note that the proofs of the soliton resolution results above are motivated by works on soliton resolution for critical wave equations; see, for example, \cite{DuyckaertsKenigMerle2013Camb,Cote2015CPAMsolitonResol,JiaKenig2017AJMwaveSolResol,DJKM2017GaFASolresolSequence,DuyckaertsKenigMerle2023Acta,JendrejLawrie2025JAMS,JendrejLawrie2023AnnPDESolResol}. Without symmetry, a soliton resolution result for \eqref{eq:HMHF} appears in \cite{JLS2025Pi}.

\subsection{Strategy of the proof}
We use the unified notation of Section~\ref{sec:notation and pre}. For the first reading, we recommend the readers to consider only the case of \eqref{eq:NLH}. Assume the ranges of Theorems~\ref{thm:HMHF main} and \ref{thm:NLH main}, i.e.,
\begin{equation}\label{eq:dimension parameter}
	d=2D+2,\qquad 0<D<2.
\end{equation}
Let $u$ be a finite-energy solution to \eqref{eq:HMHF} or a $\dot H^1$-bounded solution to \eqref{eq:NLH}. Apply soliton resolution (Propositions~\ref{prop:sol resol hmhf} and \ref{prop:sol resol nlh}) to obtain a bubble decomposition of $u$. We need to prove that the number of bubbles is at most one, i.e., $N\leq1$.

We argue by contradiction. Suppose
\begin{equation*}
	N\geq2.
\end{equation*}
Choosing suitable $C^1$ modulation scales, we write
\begin{gather*}
	\min\{\sqrt{T_+-t},\sqrt t\}\gg\lambda_1(t)\gg\cdots\gg\lambda_N(t),
	\\
	u(t)=U(t)+g(t),\quad U(t)=\sum_{j=1}^N\iota_jW_{\lambda_j}.
\end{gather*}
Here $U$ is the pure $N$-bubble profile and $g$ is the remainder. Compared to the statement of the soliton resolution propositions, our $g$ also contains the time-independent term $u^*$, while $u^*=0$ in the global case. Crucial ingredients easily drawn from our assumption $N\geq2$ are
\begin{equation}\label{eq:intro sub self similar}
	\lambda_N(t)\ll\lambda_{N-1}(t)\ll\min\{\sqrt{T_+-t},\sqrt t\}.
\end{equation}

Motivated by the formal modulation viewpoint of \cite{KimMerle2025CPAM}, one may formally apply the standard argument to derive an evolution law for $\lambda_N$ in the present setting. In the global case $T_+=\infty$, the resulting Riccati-type law suggests that a scale-separated bubble tree cannot persist for all forward times. When $T_+<\infty$, however, the same argument does not by itself exclude a bubble tree, since the Riccati variable may become singular precisely at $T_+$ as $\lambda_N(t)\to0$. We will instead study the relative dynamics of $\lambda_{N-1}$ and $\lambda_N$ and derive a contradiction from their asymptotic behavior as $t\to T_+$.

We work with the two innermost scales because every omitted bubble then lies at a larger scale, so its contribution is lower order in the relevant estimates, such as \eqref{eq:interaction main} and \eqref{eq:HU ThetaN}. If one instead used $\lambda_1$ and $\lambda_2$ when $N>2$, the smaller bubbles $W_{\lambda_j}$ with $j\geq3$ would enter these estimates and need not be lower order without additional relations among the consecutive scale ratios, which the soliton resolution does not provide.

A natural approach is to compare the corresponding scaling directions. The single-scale projection used in the high-dimensional setting is not directly available here without localization, since $r\Lambda W\notin L^2_d$. A cutoff could be introduced, but the resulting errors may be difficult to control; see \cite{KimKimKwon2024arxiv} for a related discussion. We instead cancel the common leading tail of the two normalized scaling directions. Indeed, for some nonzero constant $c\neq 0$,
\begin{equation*}
	\lambda^{2-D}(\Lambda W)_{\underline{\lambda}}(r)
	=\lambda^{-D}(\Lambda W)_\lambda(r)
	= (c+o(1))r^{-2D}\quad\text{as}\quad r\to\infty.
\end{equation*}
Thus, the leading-order decay is independent of $\lambda$. We therefore define
\begin{equation*}
	\Theta_N\coloneqq\iota_{N-1}(\lambda_N^{2-D}(\Lambda W)_{\underline{\lambda_N}}
	-\lambda_{N-1}^{2-D}(\Lambda W)_{\underline{\lambda_{N-1}}}).
\end{equation*}
The common leading tail cancels, so $r\Theta_N\in L^2_d$, while $\Theta_N$ retains the interaction between the bubbles at $\lambda_{N-1}$ and $\lambda_N$.

We now give the formal calculation that motivates the argument in the present setting. Testing $U_t\approx\calT(U)$ against $\Theta_N$, and formally treating the contribution of $g$ as an error, gives
\begin{equation}\label{eq:intro formal law 0}
	(\Theta_N,U_t)_d=(\Theta_N,\calT(U))_d+\text{error}.
\end{equation}
Since each bubble $\iota_jW_{\lambda_j}$ is stationary, $\calT(U)$ consists of the interactions between different bubbles. By scale separation, the dominant contribution comes from the interaction between the bubbles at $\lambda_{N-1}$ and $\lambda_N$, while all other interactions are errors. Computing this contribution yields
\begin{equation}\label{eq:intro formal law 1}
	(\Theta_N,\calT(U))_d=-\frac{\kappa}{\lambda_{N-1}^D}+\text{error},
\end{equation}
where $\kappa\neq0$ is a universal constant defined in \eqref{eq:kappa def}. On the other hand, a direct calculation gives, schematically,
\begin{equation}\label{eq:intro formal law 2}
	(\Theta_N,U_t)_d
	=\frac{d}{dt}O(\lambda_{N-1}^{2-D})+\text{error}.
\end{equation}
Substituting \eqref{eq:intro formal law 1} and \eqref{eq:intro formal law 2} into \eqref{eq:intro formal law 0} yields the key differential inequality
\begin{equation*}
	\left|\frac{d}{dt}O(\lambda_{N-1}^{2-D})\right|\sim\lambda_{N-1}^{-D}.
\end{equation*}
We will show that this is incompatible with the sub-self-similar bound \eqref{eq:intro sub self similar} for $\lambda_{N-1}$ obtained from the soliton resolution.

Suppose first that $T_+<\infty$. Integrating this formal law toward $T_+$ gives
\begin{equation*}
	\lambda_{N-1}^{2-D}(t)\gtrsim\int_t^{T_+}\lambda_{N-1}^{-D}(\tau)d\tau.
\end{equation*}
Applying $\lambda_{N-1}(t)\ll\sqrt{T_+-t}$, we have
\begin{equation*}
	(T_+-t)^{\frac{2-D}{2}}\gg\lambda_{N-1}^{2-D}(t)
	\gtrsim\int_t^{T_+}\lambda_{N-1}^{-D}d\tau
	\gg\int_t^{T_+}(T_+-\tau)^{-\frac{D}{2}}d\tau
	\sim(T_+-t)^{\frac{2-D}{2}},
\end{equation*}
which is a contradiction. Here $0<D$ is used to compare the inverse powers, while $D<2$ gives the last estimate.

When $T_+=\infty$, integrating from a sufficiently large fixed time $T_0$ to $t$ and using $\lambda_{N-1}(t)\ll\sqrt t$, the same argument gives
\begin{equation*}
	t^{\frac{2-D}{2}}\gg\lambda_{N-1}^{2-D}(t)+1
	\gtrsim\int_{T_0}^t\lambda_{N-1}^{-D}d\tau
	\gg\int_{T_0}^t\tau^{-\frac{D}{2}}d\tau
	\sim t^{\frac{2-D}{2}},
\end{equation*}
which is again impossible.

Finally, to make the formal argument rigorous, we introduce
\begin{equation*}
	\begin{aligned}
		\calF_N(t)&\coloneqq\iota_N\iota_{N-1}\lambda_{N-1}^{2-D}
		\int_0^{\lambda_N/\lambda_{N-1}}(s^{-D}(\Lambda W)_s-\Lambda W,s^{-1}(\Lambda W)_s)_d ds,\\
		\mathfrak b_N(t)&\coloneqq(\Theta_N(t),g(t))_d.
	\end{aligned}
\end{equation*}
With this definition, the precise form of \eqref{eq:intro formal law 2} is
\begin{equation*}
	-(\Theta_N,U_t)_d=\partial_t\calF_N+\text{error}.
\end{equation*}
Morover, Proposition~\ref{prop:modulation estimates} provides the bound
\begin{equation*}
	|\calF_N|+|\mathfrak b_N|
	\lesssim\lambda_{N-1}^{2-D}
\end{equation*}
and the evolution law
\begin{equation*}
	\partial_t(\calF_N-\mathfrak b_N)
	=\frac{\kappa}{\lambda_{N-1}^D}+\text{error},
\end{equation*}
where the error is negligible after time integration. The preceding contradiction argument applies to $\calF_N-\mathfrak b_N$ and excludes $N\geq2$ in both the finite-time blow-up and global cases.

\vspace{5pt}
\noindent\mbox{\textbf{Acknowledgements.~}}The author would like to thank Uihyeon Jeong for helpful discussions at an early stage of this work, and Kihyun Kim for valuable discussions and helpful comments on the manuscript. The author is supported by a KIAS Individual Grant (MG105201) at Korea Institute for Advanced Study.

\section{Notation and preliminaries}\label{sec:notation and pre}

\subsection{Notation}\label{sec:notation}
We collect the notation used for the two models and multi-bubbles in a single list. 

\begin{itemize}
	\item For quantities $A\in\bbR$ and $B\geq0$, we write $A\lesssim B$ if $|A|\leq CB$ holds for some implicit constant $C$. For $A,B\geq0$, we write $A\sim B$ if $A\lesssim B$ and $B\lesssim A$. If $C$ is allowed to depend on some parameters, then we write them as subscripts of $\lesssim,\sim,\gtrsim$ to indicate the dependence.
	
	\item If $A$ is a statement, then $\chf_A$ is $1$ when $A$ is true and $0$ otherwise. If $A$ is a set, then $\chf_A$ denotes its indicator function.
	
	\item Fix a smooth cut-off function $\chi\in C_c^\infty([0,\infty))$ such that $\chi(r)=1$ for $|r|\leq1$, and $\chi(r)=0$ for $|r|\geq2$. For $R>0$, write $\chi_R(r)=\chi(r/R)$, and write $\chi_{\infty}=1$.
	
	\item We use little-$o$ notation such as $o_{t\to T_+}(1)$ for a quantity tending to zero as $t\to T_+$.
	
	\item For radial functions on $\bbR^d$ and $1\leq p<\infty$, write
	\begin{equation*}
		(h_1,h_2)_d\coloneqq\int_0^\infty h_1(r)h_2(r)r^{d-1}dr, \qquad \|h\|_{L^p_d}\coloneqq\left(\int_0^\infty|h(r)|^pr^{d-1}dr\right)^{1/p}.
	\end{equation*}
	We also use $\dot H^s_d=\dot H^s(\bbR^d)$ for the standard homogeneous Sobolev space. 
	
	\item We write the two equations in a common form using the effective dimension $d$ and the parameter $D$, where
	\begin{equation*}
		(d,D)\in\left\{\left(3,\frac12\right),(4,1),\left(5,\frac32\right)\right\}, \qquad d=2D+2.
	\end{equation*}
	For \eqref{eq:HMHF}, we use the invariance under $v\mapsto v-\ell\pi$ to take $\ell=0$. We then take $(d,D)=(4,1)$ and define
	\begin{equation*}
		\begin{gathered}
			u=r^{-1}v, \quad Q(r)=2\arctan r, \quad W=r^{-1}Q,
			\quad f(z)=z-\tfrac12\sin(2z).
		\end{gathered}
	\end{equation*}
	Here $d=4$ is the dimension of the radial function $u=r^{-1}v$, not the spatial dimension of the original harmonic map flow.
	For \eqref{eq:NLH}, retain \eqref{eq:dimension parameter} and define
	\begin{equation*}
		Q=r^DW, \qquad f(z)=|z|^{2/D}z.
	\end{equation*}
	In either model, $Q=r^DW$.
	\item Both equations take the unified form $u_t=\calT(u)$. For radial functions $\phi$ and $h$ on $\bbR^d$, define
	\begin{equation}\label{eq:NL def}
		\begin{aligned}
			\calT(h)&\coloneqq\Delta_dh+r^{-(D+2)}f(r^Dh), \qquad
			H_\phi\coloneqq-\Delta_d-r^{-2}f'(r^D\phi),\\
			\NL_\phi(h)&\coloneqq\calT(\phi+h)-\calT(\phi)+H_\phi h\\
			&=r^{-D-2}[f(r^D(\phi+h))-f(r^D\phi)-f'(r^D\phi)r^Dh].
		\end{aligned}
	\end{equation}
	Here, $\Delta_d\coloneqq\Delta_{\bbR^d}=\partial_{rr}+(d-1)r^{-1}\partial_r$ on radial functions. Also, we have
	\begin{equation}\label{eq:error equ}
		\calT(\phi+h)=\calT(\phi)-H_\phi h+\NL_\phi(h).
	\end{equation}
	
	\item For a radial function $\phi$ on $\bbR^d$ and $\lambda>0$, define
	\begin{equation*}
		\phi_\lambda(r)\coloneqq\lambda^{-D}\phi(r/\lambda), \qquad \Lambda\phi\coloneqq(D+r\partial_r)\phi.
	\end{equation*}
	Once this $\dot H^1$-critical scaling and its generator are fixed, define
	\begin{equation*}
		\phi_{\underline\lambda}(r)\coloneqq\lambda^{-2}\phi_\lambda(r), \qquad \Lambda_{-1}\phi\coloneqq(\Lambda+2)\phi, \qquad \lambda\partial_\lambda\phi_{\underline\lambda}=-(\Lambda_{-1}\phi)_{\underline\lambda}.
	\end{equation*}
	For $Q=r^DW$, write
	\begin{equation*}
		Q_{[\lambda]}(r)\coloneqq Q(r/\lambda)=r^DW_\lambda(r).
	\end{equation*}
	\item The stationary profile satisfies
	\begin{equation*}
		\Delta_dW+r^{-(D+2)}f(r^DW)=0, \qquad \calT(W_\lambda)=0.
	\end{equation*}
	Recall that $W(r)=\frac{2\arctan r}{r}$ for \eqref{eq:HMHF} and $W(r)=(1+\frac{r^2}{4D(D+1)})^{-D}$ for \eqref{eq:NLH}, and that $Q(r)=r^DW(r)$ in either model.
	For both models and every $r,\lambda>0$,
	\begin{equation}\label{eq:profile pointwise}
		|W_\lambda(r)|\lesssim
		\begin{cases}
			\lambda^{-D},&0<r\leq\lambda,\\
			r^{-D},&\lambda\leq r,\quad\text{for \eqref{eq:HMHF}},\\
			\lambda^Dr^{-2D},&\lambda\leq r,\quad\text{for \eqref{eq:NLH}},
		\end{cases}
	\end{equation}
	
	\item Recall also that $\Lambda W(r)=\frac{2}{1+r^2}$ for \eqref{eq:HMHF} and $\Lambda W(r)=(D-\frac{r^2}{4(D+1)})(1+\frac{r^2}{4D(D+1)})^{-D-1}$ for \eqref{eq:NLH}.
	From this, we obtain
	\begin{equation}\label{eq:resonance asymptotics}
		\begin{aligned}
			\Lambda W(r)&=
			\begin{cases}
				2+O(r^2),&\text{for \eqref{eq:HMHF}},\\
				D+O(r^2),&\text{for \eqref{eq:NLH}},
			\end{cases}
			\quad r\to0,
			\\
			\Lambda W(r)&=
			\begin{cases}
				2r^{-2}+O(r^{-4}),&\text{for \eqref{eq:HMHF}},\\
				-D[4D(D+1)]^Dr^{-2D}+O(r^{-2D-2}),&\text{for \eqref{eq:NLH}},
			\end{cases}
			\quad r\to\infty.
		\end{aligned}
	\end{equation}
	For both models and every $r,\lambda>0$,
	\begin{equation}\label{eq:Lambda W pointwise}
		|(\Lambda W)_\lambda(r)|\lesssim
		\begin{cases}
			\lambda^{-D},&0<r\leq\lambda,\\
			\lambda^Dr^{-2D},&\lambda\leq r.
		\end{cases}
	\end{equation}
	For \eqref{eq:HMHF}, we also have $r(\Lambda W)_\lambda=\sin Q_{[\lambda]}$.
	\item The exact one-bubble potential is
	\begin{equation}\label{eq:one bubble potential explicit}
		r^{-2}f'(Q(r))=
		\begin{cases}
			\frac{8}{(1+r^2)^2},&\text{for \eqref{eq:HMHF}},\\
			\frac{16D(D+1)^2(D+2)}{(r^2+4D(D+1))^2},&\text{for \eqref{eq:NLH}}.
		\end{cases}
	\end{equation}
	Define the one-bubble linearized operators by
	\begin{equation*}
		H\coloneqq-\Delta_d-r^{-2}f'(Q), \qquad H_\lambda\coloneqq-\Delta_d-r^{-2}f'(Q_{[\lambda]}).
	\end{equation*}
	Their scaling kernels satisfy
	\begin{equation*}
		H(\Lambda W)=H_\lambda((\Lambda W)_\lambda)=0.
	\end{equation*}
	\item For $2\leq j\leq N$, define
	\begin{equation*}
		U\coloneqq\sum_{j=1}^N\iota_jW_{\lambda_j}, \quad u=U+g, \quad P\coloneqq r^DU=\sum_{j=1}^N\iota_jQ_{[\lambda_j]}, \quad \mu_j\coloneqq \frac{\lambda_j}{\lambda_{j-1}}.
	\end{equation*}
	We view $u,W,U,g$ as radial functions on $\bbR^d$. The variables $Q$ and $P=r^DU$ are used in the nonlinear expressions involving $f$ and $f'$. 
	
	\item Fix $R_0>10$ and choose a radial function $Z\in C_c^\infty((0,\infty))$ satisfying
	\begin{equation*}
		(Z,\Lambda W)_d=1, \qquad \operatorname{supp}Z\subset[R_0^{-1},R_0], \qquad \int_0^\infty Z(r)rdr=0.
	\end{equation*}
	\item The interaction constant used in the modulation estimates is
	\begin{equation*}
		\kappa\coloneqq-W(0)(\Lambda W,r^{-2}f'(Q))_d=
		\begin{cases}
			-8,&\text{for \eqref{eq:HMHF}},\\
			2^{2D+1}D^{D+2}(D+1)^D,&\text{for \eqref{eq:NLH}}.
		\end{cases}
	\end{equation*}
\end{itemize}
\subsection{Preliminary estimates}

Recall that $d=2D+2$.

\begin{lem}
	The following inequalities hold.
	\begin{enumerate}
		\item We have
		\begin{equation}\label{eq:radial GN}
			\|\phi\|_{L^4_2}^2\lesssim\|\phi\|_{L^2_2}\|\phi_r\|_{L^2_2}.
		\end{equation}
		\item For $d\in\{3,4,5\}$, we have
		\begin{equation}\label{eq:Hardy Sobolev}
			\|r^{-1}\phi\|_{L^2_d}
			+\|\phi\|_{L^{2d/(d-2)}_d}\lesssim_d\|\phi\|_{\dot H^1_d}.
		\end{equation}
		\item For $d=3$, we have
		\begin{equation}\label{eq:radial L inf}
			|\phi(r)|\lesssim r^{-1/2}\|\phi_r\|_{L^2_3(r,\infty)},
			\qquad
			\|\phi\|_{L^\infty}^2\lesssim\|\phi\|_{\dot H^1_3}\|\Delta_3\phi\|_{L^2_3}
		\end{equation}
	\end{enumerate}
\end{lem}

\begin{proof}
	Assertions (1) and (2) are the radial Gagliardo--Nirenberg inequality on $\bbR^2$ and the Hardy and Sobolev inequalities on $\bbR^d$. For (3),
	\begin{equation*}
		|\phi(r)|\leq\|\phi_r\|_{L^2_3(r,\infty)}({\textstyle\int_r^\infty} s^{-2}ds)^{1/2}.
	\end{equation*}
	For every $R>0$, apply Cauchy--Schwarz on $\{|\xi|\leq R\}$ and $\{|\xi|>R\}$:
	\begin{equation*}
		\|\phi\|_{L^\infty}
		\lesssim R^{1/2}\|\phi\|_{\dot H^1_3}+R^{-1/2}\|\Delta_3\phi\|_{L^2_3}.
	\end{equation*}
	Choose $R=\|\Delta_3\phi\|_{L^2_3}/\|\phi\|_{\dot H^1_3}$.
\end{proof}

\begin{lem}
	Let $D\in\{\frac12,1,\frac32\}$ and $d=2D+2$. The following inequalities hold:
	\begin{align}
		\|(r^Dh_r)_r\|_{L^2_2}+\|r^{-1}h_r\|_{L^2_d}
		&\lesssim \|\Delta_dh\|_{L^2_d},\label{eq:second order Hardy}\\
		\|r^Dh\|_{L^\infty}&\lesssim \|h\|_{\dot H^1_d}.\label{eq:radial Linfty}
		\\
		\int_0^\infty h^2r^{D-1}dr &\lesssim \|h\|_{\dot H^1_d}^{2-D}\|h\|_{\dot H^2_d}^D. \label{eq:weighted interpolation}
	\end{align}
	If $D\in\{1,\frac32\}$, then
	\begin{align}
		\|r^{D-1}h\|_{L^4_2}^2=\|r^{D-2}h^2\|_{L^2_d}&\lesssim \|r^Dh_r\|_{L^4_2}^2
		\lesssim _D\|h\|_{\dot H^1_d}\|\Delta_dh\|_{L^2_d}. \label{eq:first order L4}
	\end{align}
\end{lem}
\begin{proof}
	Let $D\in\{\frac12,1,\frac32\}$.
	For \eqref{eq:second order Hardy}, integrating by parts, we find
	\begin{equation*}
		\|\Delta_dh\|_{L^2_d}^2
		=\|(r^Dh_r)_r\|_{L^2_2}^2+(D+1)^2\|r^{D-1}h_r\|_{L^2_2}^2.
	\end{equation*}
	Furthermore, for \eqref{eq:radial Linfty},
	\begin{equation*}
		r^{2D}|h(r)|^2\leq2\int_r^\infty|h h_r|s^{2D}ds\lesssim \|h\|_{\dot H^1_d}^2
	\end{equation*}
	by \eqref{eq:Hardy Sobolev}. To show \eqref{eq:weighted interpolation}, integrate by parts and apply Cauchy--Schwarz:
	\begin{equation*}
		D\int_0^\infty h^2r^{D-1}dr
		=-2\int_0^\infty hh_rr^Ddr
		\lesssim \left(\int_0^\infty h^2r^{D-1}dr\right)^{1/2}
		\left(\int_0^\infty h_r^2r^{D+1}dr\right)^{1/2}.
	\end{equation*}
	Interpolating between the weights $r^{2D+1}$ and $r^{2D-1}$, and using \eqref{eq:second order Hardy}, we obtain
	\begin{equation*}
		\int_0^\infty h_r^2r^{D+1}dr
		\leq\|h\|_{\dot H^1_d}^{2-D}\|r^{-1}h_r\|_{L^2_d}^D
		\lesssim \|h\|_{\dot H^1_d}^{2-D}\|\Delta_dh\|_{L^2_d}^D.
	\end{equation*}
	Cancelling the term $\left(\int_0^\infty h^2r^{D-1}dr\right)^{1/2}$, we arrive at \eqref{eq:weighted interpolation}.

	Let $D\in\{1,\frac32\}$. For the first inequality of \eqref{eq:first order L4}, let $r=e^x$ and $G(x)=e^{(D-1/2)x}h(e^x)$. Then
	\begin{equation*}
		\|r^{D-1}h\|_{L^4_2}=\|G\|_{L^4_x}, \qquad
		\|r^Dh_r\|_{L^4_2}=\|[-\partial_x+(D-\tfrac12)]G\|_{L^4_x}.
	\end{equation*}
	For $c=D-\frac12>0$, we have $\partial_y(e^{-c(y-x)}G(y))=-e^{-c(y-x)}(-\partial_y+c)G(y)$, or
	\begin{equation*}
		G(x)=\int_x^\infty e^{-c(y-x)}(-\partial_y+c)G(y)dy.
	\end{equation*}
	Young's inequality now implies \eqref{eq:first order L4}. Next, applying \eqref{eq:radial GN} and \eqref{eq:second order Hardy}, we conclude the second inequality of \eqref{eq:first order L4}.
\end{proof}

\begin{lem}[Coercivity estimate for $H$]\label{lem:H coercivity}
	Let $Z$ be fixed by \eqref{eq:Z conditions}. For every $\lambda>0$, if $(Z_{\underline\lambda},h)_d=0$, then
	\begin{equation}\label{eq:H coercivity}
		\|h\|_{\dot H^2_d}+\|r^{-2}f'(Q_{[\lambda]})h\|_{L^2_d}
		\lesssim\|H_\lambda h\|_{L^2_d},
	\end{equation}
	with a constant independent of $\lambda$.
\end{lem}
\begin{proof}
	By scaling, it is enough to consider $\lambda=1$. Indeed,
	\begin{equation*}
		(Z_{\underline\lambda},h)_d=(Z,\lambda^Dh(\lambda\cdot))_d,\qquad
		\|h\|_{\dot H^2_d}=\lambda^{-1}\|-\Delta_d(\lambda^Dh(\lambda\cdot))\|_{L^2_d},
	\end{equation*}
	and
	\begin{equation*}
		\|H_\lambda h\|_{L^2_d}=\lambda^{-1}\|H(\lambda^Dh(\lambda\cdot))\|_{L^2_d}.
	\end{equation*}
	We prove the estimate by contradiction. Otherwise, there exist radial functions $h_n$ such that
	\begin{equation*}
		(Z,h_n)_d=0,\qquad
		\|h_n\|_{\dot H^2_d}=1,\qquad
		\|Hh_n\|_{L^2_d}\to0.
	\end{equation*}
	Since $H=-\Delta_d-r^{-2}f'(Q)$, we have
	\begin{equation}\label{eq:H potential mass}
		1\gtrsim\|r^{-2}f'(Q)h_n\|_{L^2_d}
		\geq\|h_n\|_{\dot H^2_d}-\|Hh_n\|_{L^2_d}
		=1-o_n(1).
	\end{equation}
	We claim that $r^{-2}f'(Q)h_n$ is precompact in $L^2_d$. For every $R>1$, \eqref{eq:one bubble potential explicit} implies
	\begin{equation*}
		r^{-2}f'(Q(r))\gtrsim_R1,\qquad 0<r<2R.
	\end{equation*}
	Hence,
	\begin{equation*}
		\|\chf_{\{r<2R\}}h_n\|_{L^2_d}
		\lesssim_R\|r^{-2}f'(Q)h_n\|_{L^2_d}
		\lesssim_R1.
	\end{equation*}
	From $\|\Delta_dh_n\|_{L^2_d}=1$ and the Rellich--Kondrachov theorem, we obtain the local precompactness of $r^{-2}f'(Q)h_n$.
	
	By averaging on $r\sim1$, \eqref{eq:Hardy Sobolev}, and \eqref{eq:second order Hardy}, we obtain
	\begin{equation*}
		\|\chf_{\{r>1\}}r^{-3}h_n\|_{L^2_d}^2
		\lesssim\|\chf_{\{r\sim1\}}h_n\|_{L^2_d}^2
		+\|\chf_{\{r>1\}}r^{-2}\partial_rh_n\|_{L^2_d}^2
		\lesssim1.
	\end{equation*}
	Since $r^{-2}f'(Q(r))\lesssim r^{-4}$ for $r\geq1$, we have
	\begin{equation*}
		\|\chf_{\{r>R\}}r^{-2}f'(Q)h_n\|_{L^2_d}^2
		\lesssim R^{-2}\|\chf_{\{r>1\}}r^{-3}h_n\|_{L^2_d}^2
		\lesssim R^{-2}.
	\end{equation*}
	Since $R>1$ is arbitrary, $r^{-2}f'(Q)h_n$ is precompact in $L^2_d$. Thus, passing to a subsequence, we obtain
	\begin{equation*}
		h_n\rightharpoonup h\quad\text{in}\quad\dot H^2_d,\qquad
		r^{-2}f'(Q)h_n\to r^{-2}f'(Q)h\quad\text{in}\quad L^2_d.
	\end{equation*}
	In particular, \eqref{eq:H potential mass} implies $h\neq0$, and $(Z,h)_d=0$. On the other hand, for every radial $\varphi\in C_c^\infty(0,\infty)$,
	\begin{equation*}
		(\varphi,Hh)_d
		=\lim_{n\to\infty}(\varphi,Hh_n)_d=0.
	\end{equation*}
	Thus, $Hh=0$. Since the radial kernel of $H$ is spanned by $\Lambda W$ and $(Z,\Lambda W)_d=1$, the orthogonality $(Z,h)_d=0$ forces $h=0$, contradicting $h\neq0$. Therefore,
	\begin{equation*}
		\|h\|_{\dot H^2_d}\lesssim\|H_\lambda h\|_{L^2_d}.
	\end{equation*}
	Finally, from the definition of $H_\lambda$, we arrive at
	\begin{equation*}
		\|r^{-2}f'(Q_{[\lambda]})h\|_{L^2_d}
		\leq\|h\|_{\dot H^2_d}+\|H_\lambda h\|_{L^2_d}
		\lesssim\|H_\lambda h\|_{L^2_d}. \qedhere
	\end{equation*}
\end{proof}

\begin{lem}[Nonlinear estimates]\label{lem:nonlinear estimates}
	For \eqref{eq:HMHF}, $a,b,a_1,\ldots,a_J\in\bbR$, and $J\geq2$, we have
	\begin{equation}\label{eq:HMHF nonlinear esti}
		\begin{aligned}
			|f'(a+b)-f'(a)|&\lesssim  |\sin b|(|\sin a|+|\sin b|),\\
			|f(a+b)-f(a)-f'(a)b|&\lesssim|b|^2,\\
			|f(a+b)-f(a)-f(b)|&\lesssim|\sin a||\sin b|(|\sin a|+|\sin b|),\\
			\bigg|f\bigg(\sum_{j=1}^Na_j\bigg)-\sum_{j=1}^Nf(a_j)\bigg|&\lesssim\sum_{1\leq j<k\leq N}|\sin a_j||\sin a_k|(|\sin a_j|+|\sin a_k|).
		\end{aligned}
	\end{equation}
	For \eqref{eq:NLH}, $D\in\{\frac12,1,\frac32\}$, $a,b,a_1,\ldots,a_J\in\bbR$, and $J\geq2$, we have
	\begin{equation}\label{eq:NLH nonlinear esti}
		\begin{aligned}
			|f'(a+b)-f'(a)|&\lesssim |a|^{2/D-1}|b|+|b|^{2/D},\\
			|f(a+b)-f(a)-f'(a)b|&\lesssim |a|^{2/D-1}|b|^2+|b|^{1+2/D},\\
			|f(a+b)-f(a)-f(b)|&\lesssim |a|^{2/D}|b|+|b|^{2/D}|a|,\\
			\bigg|f\bigg(\sum_{j=1}^N a_j\bigg)-\sum_{j=1}^Nf(a_j)\bigg|&\lesssim \sum_{1\leq j<k\leq N}\left(|a_j|^{2/D}|a_k|+|a_k|^{2/D}|a_j|\right).
		\end{aligned}
	\end{equation}
\end{lem}

\begin{proof}
	For \eqref{eq:HMHF}, the identities
	\begin{equation*}
		f'(a+b)-f'(a)=2\sin(2a+b)\sin b, \quad f(a+b)-f(a)-f(b)=2\sin(a+b)\sin a\sin b,
	\end{equation*}
	together with the boundedness of $f''$, prove the first three inequalities.
	
	For \eqref{eq:NLH}, the first three inequalities follow from
	\begin{equation*}
		||a+b|^{2/D}-|a|^{2/D}|\lesssim |a|^{2/D-1}|b|+|b|^{2/D}
	\end{equation*}
	and the fundamental theorem of calculus. In both models, the last inequality follows by induction from the third inequality, using $|\sin(\sum_{j=1}^Na_j)|\leq\sum_{j=1}^N|\sin a_j|$ for \eqref{eq:HMHF} and $|\sum_{j=1}^Na_j|^{2/D}\lesssim\sum_{j=1}^N|a_j|^{2/D}$ for \eqref{eq:NLH}.
\end{proof}

\section{Modulation analysis}
In this section, we establish the modulation estimates and use them to rule out bubble trees with $N\geq2$. Unless otherwise stated, we treat \eqref{eq:HMHF} and \eqref{eq:NLH} simultaneously in the unified notation of Section~\ref{sec:notation and pre}, with $D=1$ for \eqref{eq:HMHF} and $D\in\{\frac12,1,\frac32\}$ for \eqref{eq:NLH}. 

Let $u(t)$ denote the solution under consideration in the unified variables, and let $T_+\in(0,\infty]$ be its maximal forward lifespan. If $T_+<\infty$, define $u^*=r^{-1}v^*$ for \eqref{eq:HMHF}, while for \eqref{eq:NLH} let $u^*\in\dot H^1_d$ be the asymptotic profile supplied by Proposition~\ref{prop:sol resol nlh}. If $T_+=\infty$, define $u^*=0$ and interpret $T_+-t=\infty$. With these conventions, Propositions~\ref{prop:sol resol hmhf} and \ref{prop:sol resol nlh} yield an integer $N\geq1$ if $T_+<\infty$ and $N\geq0$ if $T_+=\infty$, signs $\iota_1,\ldots,\iota_N\in\{\pm1\}$, continuous positive scales $\wt\lambda_1(t),\ldots,\wt\lambda_N(t)$, and $\eps(t)\in\dot H^1_d$ such that
\begin{equation}\label{eq:bubble decomposition}
	u(t)=u^*+\sum_{j=1}^N\iota_jW_{\wt\lambda_j(t)}+\eps(t)
\end{equation}
with
\begin{equation}\label{eq:decomposition convergence}
	\|\eps(t)\|_{\dot H^1_d}
	+\sum_{j=2}^N\frac{\wt\lambda_j(t)}{\wt\lambda_{j-1}(t)}
	+\frac{\wt\lambda_1(t)}{\min\{\sqrt{T_+-t},\sqrt t\}}\to0
	\quad\text{as}\quad t\to T_+.
\end{equation}
The energy identities for the two models also imply
\begin{equation}\label{eq:dissipation}
	\int_0^{T_+}\|\calT(u(t))\|_{L^2_d}^2dt<\infty,
\end{equation}
where $\calT$ is defined in \eqref{eq:NL def}.
When $N=0$, the sum in \eqref{eq:bubble decomposition} is zero and all scale terms in \eqref{eq:decomposition convergence} are omitted. In the finite-time \eqref{eq:HMHF} case, $u^*=r^{-1}v^*$ need not belong to $\dot H^1_4$; only its cutoff near the origin will be measured in that norm.

To prove the main theorems, we assume for contradiction that $N\geq2$ and fix a sufficiently small $0<\eta \ll1$. Since the scales $\wt\lambda_1,\ldots,\wt\lambda_N$ in \eqref{eq:bubble decomposition} are only continuous, we use orthogonality conditions to fix nearby $C^1$ scales. For this purpose, fix $R_0>10$ and choose a radial function $Z\in C_c^\infty((0,\infty))$ satisfying
\begin{equation}\label{eq:Z conditions}
	(Z,\Lambda W)_d=1, \qquad \operatorname{supp}Z\subset[R_0^{-1},R_0], \qquad \int_0^\infty Z(r)rdr=0.
\end{equation}
In the global case, we use the conventions $r_0=\infty$, $\chi_{r_0}=1$, and $C_{r_0}=0$. For the following lemma, define
\begin{equation}\label{eq:delta}
	\delta(t)\coloneqq\max\{\|\chi_{2r_0}g(t)\|_{\dot H^1_d}^2,\mu_N(t)^D\}.
\end{equation}

\begin{lem}[Decomposition]\label{lem:modulation}
	Fix a sufficiently small $0<\eta\ll1$. Then there exist $r_0\in(0,\infty]$, $t_1^*\in(0,T_+)$, and $\lambda_1,\ldots,\lambda_N \in C^1$ with the following properties. 
	\begin{itemize}
		\item (Decomposition) The solution satisfies the decomposition
		\begin{equation}\label{eq:common decomposition}
			U=\sum_{j=1}^N\iota_jW_{\lambda_j}, \qquad g=u-U
		\end{equation}
		with the orthogonality condition
		\begin{equation}\label{eq:orthogonality}
			(Z_{\underline{\lambda_j}},g)_d=0, \qquad 1\leq j\leq N.
		\end{equation}
		
		\item (Scale decoupling) The scales satisfy 
		\begin{equation}\label{eq:modulation smallness}
			\max_{1\leq j\leq N}\bigg|\log\bigg(\frac{\lambda_j}{\wt\lambda_j}\bigg)\bigg|
			+\sum_{j=2}^N\mu_j(t)
			+\frac{\lambda_1(t)}{\min\{\sqrt{T_+-t},\sqrt t\}}\to0
			\quad\text{as}\quad t \to T_+.
		\end{equation}
		
		\item (Smallness of remainder) We have
		\begin{equation}\label{eq:delta bound}
			\delta(t)\leq\eta^4, \qquad t_1^*\leq t<T_+.
		\end{equation}
		If $u^*=0$, then $r_0=\infty$ and
		\begin{equation}\label{eq:delta smallness}
			\delta(t)\to0\quad\text{as}\quad t\to\infty.
		\end{equation}
		
		\item (First modulation estimate) For $t\in[t_1^*,T_+)$, we have
		\begin{equation}\label{eq:lambda t esti}
			\sum_{j=1}^N|\lambda_{j,t}(t)|\lesssim \|\calT(u(t))\|_{L^2_d}
		\end{equation}
	\end{itemize}
\end{lem}

\begin{proof}
	We first choose $r_0$ so that $\|\chi_{8r_0}u^*\|_{\dot H^1_d}\ll\eta$. When $T_+=\infty$, since $u^*\equiv0$, our convention $r_0=\infty$ automatically satisfies this condition.
	
	Define a map 
	\begin{equation*}
		(\lambda_1,\lambda_2,\cdots,\lambda_N)\mapsto
		\mathbf F =(\mathbf F_1,\mathbf F_2,\cdots, \mathbf F_N),
	\end{equation*}
	where
	\begin{equation*}
		\mathbf F_{k} (\lambda_1,\lambda_2,\cdots,\lambda_N)\coloneqq(\iota_kZ_{\underline{\lambda_k}},u-(1-\chi_{4r_0})u^*-{\textstyle\sum_{j=1}^N} \iota_jW_{\lambda_j})_d.
	\end{equation*}
	At $(\wt\lambda_1,\ldots,\wt\lambda_N)$, we claim that $\mathbf F$ vanishes as $t\to T_+$. Indeed, \eqref{eq:bubble decomposition} gives
	\begin{equation*}
		\mathbf F_{k}=(\iota_kZ_{\underline{\wt\lambda_k}},\eps+\chi_{4r_0}u^*)_d.
	\end{equation*}
	For the $\eps$ part, we have
	\begin{equation*}
		|(Z_{\underline{\wt\lambda_k}},\eps)_d| \lesssim \|\eps\|_{\dot H^1_d} \to 0.
	\end{equation*}
	For the asymptotic profile term, if $T_+=\infty$, then $u^*=0$. Suppose that $T_+<\infty$. The support of $Z$ implies
	\begin{equation*}
		|(Z_{\underline{\wt\lambda_k}},\chi_{4r_0}u^*)_d|\lesssim \|r^{-1}u^*\|_{L^2_d(0,R_0\wt\lambda_k)}
		\to 0\quad\text{as}\quad t\to T_+.
	\end{equation*}
	The same conclusions hold with $Z$ replaced by $\Lambda_{-1}Z$.
	
	We next compute the Jacobian. For $1\leq j,k\leq N$,
	\begin{equation*}
		\begin{aligned}
			\lambda_j\partial_{\lambda_j}\mathbf F_k
			=&(\iota_kZ_{\underline{\lambda_k}},\iota_j(\Lambda W)_{\lambda_j})_d\\
			&-\delta_{jk}
			(
			\iota_k(\Lambda_{-1}Z)_{\underline{\lambda_k}},
			u-(1-\chi_{4r_0})u^*-{\textstyle\sum_{\ell=1}^N}\iota_{\ell}W_{\lambda_{\ell}}
			)_d.
		\end{aligned}
	\end{equation*}
	For the second term, the same argument as above shows that it converges to zero as $t\to T_+$.
	Directly from \eqref{eq:resonance asymptotics} and $\int_0^\infty Z(r)rdr=0$ in \eqref{eq:Z conditions},
	\begin{equation}\label{eq:almost orthogonality}
		|(Z_{\underline{\lambda_k}},(\Lambda W)_{\lambda_j})_d-\delta_{kj}|
		\lesssim
		\begin{cases}
			(\lambda_k/\lambda_j)^D,&j<k,\\
			0,&j=k,\\
			(\lambda_j/\lambda_k)^{D+2},&j>k.
		\end{cases}
	\end{equation}
	Note that the normalization in \eqref{eq:Z conditions} yields $(\iota_kZ_{\underline{\lambda_k}},\iota_k(\Lambda W)_{\lambda_k})_d=(Z,\Lambda W)_d=1$.
	Consequently, at $(\wt\lambda_1,\ldots,\wt\lambda_N)$, the Jacobian is the identity matrix plus an error matrix tending to zero as $t\to T_+$. The implicit function theorem yields unique $C^1$ scales $\lambda_1,\ldots,\lambda_N$ satisfying
	\begin{equation*}
		\max_{1\leq j\leq N}|\log({\lambda_j}/{\wt\lambda_j})|\to 0
		\quad\text{as}\quad t\to T_+.
	\end{equation*}
	For each $t_0$ sufficiently close to $T_+$, the implicit function theorem gives $C^1$ functions $\lambda_1^{t_0},\ldots,\lambda_N^{t_0}$ on an interval $I^{t_0}\ni t_0$. If $\tau\in I^{t_1}\cap I^{t_2}$, then $(\lambda_1^{t_1}(\tau),\ldots,\lambda_N^{t_1}(\tau))$ and $(\lambda_1^{t_2}(\tau),\ldots,\lambda_N^{t_2}(\tau))$ are zeros of $\mathbf F$ at time $\tau$ and are relatively $o(1)$-close to $(\wt\lambda_1(\tau),\ldots,\wt\lambda_N(\tau))$. The uniform uniqueness in the implicit function theorem implies
	\begin{equation*}
		\lambda_j^{t_1}(\tau)=\lambda_j^{t_2}(\tau), \qquad 1\leq j\leq N.
	\end{equation*}
	Hence $\lambda_j(\tau)\coloneqq\lambda_j^{t_0}(\tau)$ for $\tau\in I^{t_0}$ is well-defined and belongs to $C^1([t_1^*,T_+))$.
	
	After increasing $t_1^*$, we have $R_0\lambda_1<2r_0$. Hence $(1-\chi_{4r_0})u^*$ vanishes on every $\operatorname{supp}Z_{\underline{\lambda_k}}$, and the equations obtained from the implicit function theorem become
	\begin{equation*}
		(Z_{\underline{\lambda_k}},g)_d=0, \qquad 1\leq k\leq N.
	\end{equation*}
	The convergence $\lambda_j/\wt\lambda_j\to1$ and \eqref{eq:decomposition convergence} establish \eqref{eq:modulation smallness}. Using \eqref{eq:bubble decomposition}, we also have
	\begin{equation*}
		\|\chi_{2r_0}g\|_{\dot H^1_d}\lesssim\|\eps\|_{\dot H^1_d}+\|\chi_{4r_0}u^*\|_{\dot H^1_d}+{\textstyle\sum_{j=1}^N}\|W_{\wt\lambda_j}-W_{\lambda_j}\|_{\dot H^1_d}.
	\end{equation*}
	The first two terms are bounded by $\eta^4$ after the choice of $r_0$ and an increase of $t_1^*$, while the last sum tends to zero because $\lambda_j/\wt\lambda_j\to1$. Increasing $t_1^*$ again, we obtain
	\begin{equation*}
		\|\chi_{2r_0}g\|_{\dot H^1_d}^2+\mu_N^D\leq\eta^4.
	\end{equation*}
	The definition \eqref{eq:delta} now establishes \eqref{eq:delta bound}. If $u^*=0$, we may take $r_0=\infty$, and every term on the right-hand side above tends to zero. Hence, \eqref{eq:delta smallness} follows.
	
	It remains to prove \eqref{eq:lambda t esti}. Differentiating $k$-th condition in \eqref{eq:orthogonality}, and then multiplying by $\lambda_k$, we obtain
	\begin{equation}\label{eq:modulation system}
		\sum_{j=1}^N\left[\frac{\lambda_k}{\lambda_j}(Z_{\underline{\lambda_k}},\iota_j(\Lambda W)_{\lambda_j})_d-\delta_{kj}((\Lambda_{-1}Z)_{\underline{\lambda_k}},g)_d\right]\lambda_{j,t}=-\lambda_k(Z_{\underline{\lambda_k}},u_t)_d.
	\end{equation}
	For $j\neq k$, \eqref{eq:almost orthogonality} implies
	\begin{equation*}
		\left|\frac{\lambda_k}{\lambda_j}(Z_{\underline{\lambda_k}},\iota_j(\Lambda W)_{\lambda_j})_d\right|\lesssim\left(\frac{\min\{\lambda_j,\lambda_k\}}{\max\{\lambda_j,\lambda_k\}}\right)^{D+1}=o_{t\to T_+}(1).
	\end{equation*}
	For $j=k$, \eqref{eq:Z conditions}, the support of $Z_{\underline{\lambda_k}}$, and \eqref{eq:delta bound} yield
	\begin{equation*}
		\left|\frac{\lambda_k}{\lambda_k}(Z_{\underline{\lambda_k}},\iota_k(\Lambda W)_{\lambda_k})_d-((\Lambda_{-1}Z)_{\underline{\lambda_k}},g)_d-1\right|\lesssim\|\chi_{2r_0}g\|_{\dot H^1_d}\leq\eta^2.
	\end{equation*}
	Thus the coefficient matrix in \eqref{eq:modulation system} is uniformly invertible after choosing $\eta$ sufficiently small and increasing $t_1^*$. Finally,
	\begin{equation*}
		\lambda_k|(Z_{\underline{\lambda_k}},u_t)_d|\lesssim\|u_t\|_{L^2_d}.
	\end{equation*}
	Since $u_t=\calT(u)$, inversion of \eqref{eq:modulation system} on $\ell^1$ establishes \eqref{eq:lambda t esti}.
\end{proof}

The preceding lemma fixes the $C^1$ scales through the orthogonality conditions $(Z_{\underline{\lambda_j}},g)_d=0$. However, when we differentiate these conditions in time, we obtain only \eqref{eq:lambda t esti}, which is not precise enough for the argument below. A natural refinement is to use the scaling direction itself and introduce a correction based on the inner product $((\Lambda W)_{\underline{\lambda}},g)_d$. Our argument follows the idea in \cite{KimMerle2025CPAM,KimMerle2026arXiv}, which extends the one-bubble argument of \cite{CollotMerleRaphael2017CMP}.

In the present range, however, the slow decay of $\Lambda W$ makes the refinement above inapplicable without further modification. As explained in the strategy of the proof, we instead seek a law for the two innermost scales by comparing their scaling directions. After normalization, these directions have the same leading tail, which cancels in their difference. We define
\begin{equation}\label{eq:Theta N def}
	\Theta_N\coloneqq\iota_{N-1}(\lambda_N^{2-D}(\Lambda W)_{\underline{\lambda_N}}
	-\lambda_{N-1}^{2-D}(\Lambda W)_{\underline{\lambda_{N-1}}}).
\end{equation}
The cancellation gives $r\Theta_N\in L^2_d$, and we use $\Theta_N$ as the test function in the refined modulation estimate.

As explained in the strategy of the proof, the principal terms in $-(\Theta_N,U_t)_d$ involve $\lambda_{N,t}$ and $\lambda_{N-1,t}$. We choose $\calF_N$ so that $\partial_t\calF_N$ captures these terms. Recall that $\mu_N=\lambda_N/\lambda_{N-1}$. Define
\begin{equation}\label{eq:calF N def}
	\calF_N(t)\coloneqq\iota_N\iota_{N-1}\lambda_{N-1}^{2-D}\int_0^{\mu_N}(s^{-D}(\Lambda W)_s-\Lambda W,s^{-1}(\Lambda W)_s)_d ds.
\end{equation}
To account for the remainder $g=u-U$, we introduce the correction formed with the same test function:
\begin{equation}\label{eq:frak b def}
	\mathfrak b_N(t)\coloneqq(\Theta_N(t),g(t))_d.
\end{equation}
Since $\partial_t\mathfrak b_N=(\Theta_N,g_t)_d+(\partial_t\Theta_N,g)_d$, the $g_t$ contribution is incorporated into $\partial_t(\calF_N-\mathfrak b_N)$, while the second term is an error.

Using $u_t=\calT(u)$ and the stationarity of each bubble, the leading contribution comes from the interaction between the bubbles at $\lambda_{N-1}$ and $\lambda_N$; the contributions involving $g$, the larger scales, and the derivatives of the scaling directions are treated as errors. The coefficient of the leading interaction is the nonzero constant
\begin{equation}\label{eq:kappa def}
	\kappa\coloneqq-W(0)(\Lambda W,r^{-2}f'(Q))_d=
	\begin{cases}
		-8,&\text{for \eqref{eq:HMHF}},\\
		2^{2D+1}D^{D+2}(D+1)^D,&\text{for \eqref{eq:NLH}}.
	\end{cases}
\end{equation}
The following proposition gives the resulting refined modulation estimate and the size bound needed in the final argument.

\begin{prop}[Modulation estimates]\label{prop:modulation estimates}
	For almost every $t\in[t_1^*,T_+)$ sufficiently close to $T_+$,
	\begin{equation}\label{eq:F b bound}
		|\calF_N|+|\mathfrak b_N|
		\lesssim\lambda_{N-1}^{2-D},
	\end{equation}
	and
	\begin{equation}\label{eq:mod esti}
		\begin{aligned}
			\left|\partial_t(\calF_N-\mathfrak b_N)-\kappa\lambda_{N-1}^{-D}\right|
			\lesssim&
			(\delta^{1/4}+o_{t\to T_+}(1))\lambda_{N-1}^{-D}
			+\|\calT(u)\|_{L^2_d}^D\\
			&+\lambda_{N-1}^{1-D}\|\calT(u)\|_{L^2_d}
			+\lambda_{N-1}^{2-D}\|\calT(u)\|_{L^2_d}^2.
		\end{aligned}
	\end{equation}
\end{prop}

To control the nonlinear part, we need a weighted estimate for $g$. The following lemma provides the required estimate.

\begin{lem}[Weighted nonlinear estimate]\label{lem:nonlinear esti}
	The following estimate holds for every $t\in[t_1^*,T_+)$ sufficiently close to $T_+$.
	\begin{equation}\label{eq:weighted esti}
		\begin{aligned}
			\int_0^{r_0}g^2r^{D-1}dr\lesssim
			\delta^{(2-D)/2}\|\calT(u)\|_{L^2_d}^D
			+\delta^{1/4}\lambda_{N-1}^{-D}.
		\end{aligned}
	\end{equation}
\end{lem}

The proof of Lemma~\ref{lem:nonlinear esti} is deferred to Section~\ref{sec:technical lem}.

\begin{proof}[Proof of Proposition~\ref{prop:modulation estimates} assuming Lemma~\ref{lem:nonlinear esti}] 
	Throughout the proof, we write $o(\cdot)$ in place of $o_{t\to T_+}(1)$.
	
	\textbf{Step 1.} We first organize $\partial_t(\calF_N-\mathfrak b_N)$. The functional $\calF_N$ is designed so that $-\partial_t\calF_N$ captures the main term of $(\Theta_N,U_t)_d$. Indeed,
	\begin{equation*}
		(\Theta_N,U_t)_d=-\partial_t\calF_N+\calR_{\calF},
	\end{equation*}
	where
	\begin{equation*}
		\begin{aligned}
			\calR_{\calF}\coloneqq&
			\lambda_{N-1}^{1-D}\lambda_{N-1,t}\bigg((2-D)\iota_N\iota_{N-1}{\int_0^{\mu_N}}(s^{-D}(\Lambda W)_s-\Lambda W,s^{-1}(\Lambda W)_s)_d ds\\
			&-\iota_N\iota_{N-1}\mu_N(\mu_N^{-D}(\Lambda W)_{\mu_N}-\Lambda W,\mu_N^{-1}(\Lambda W)_{\mu_N})_d\\
			&-(\mu_N^{-D}(\Lambda W)_{\mu_N}-\Lambda W,\Lambda W)_d\bigg)
			-\sum_{j=1}^{N-2}\lambda_{j,t}(\Theta_N,\lambda_j^{-1}\iota_j(\Lambda W)_{\lambda_j})_d.
		\end{aligned}
	\end{equation*}
	Since $u=U+g$ and $\mathfrak b_N=(\Theta_N,g)_d$, we have
	\begin{equation}\label{eq:F b derivative}
		\partial_t(\calF_N-\mathfrak b_N)=-(\Theta_N,u_t)_d+\calR_{\calF}-(\partial_t\Theta_N,g)_d.
	\end{equation}
	We next rewrite $-(\Theta_N,u_t)_d$ using the identity
	\begin{equation}\label{eq:error identity}
		u_t=\calT(u)=\calT(U)-H_Ug+\NL_U(g),
	\end{equation}
	where $H_U$ and $\NL_U$ are defined in \eqref{eq:NL def}.
	We insert this identity into \eqref{eq:F b derivative} and obtain
	\begin{equation*}
		\begin{aligned}
			\partial_t(\calF_N-\mathfrak b_N)-\kappa\lambda_{N-1}^{-D}
			=\calR\coloneqq&\calR_{\calF}
			+(\Theta_N,-\calT(U))_d-\kappa\lambda_{N-1}^{-D}\\
			&+(H_U\Theta_N,g)_d-(\Theta_N,\NL_U(g))_d-(\partial_t\Theta_N,g)_d.
		\end{aligned}
	\end{equation*}
	
	\textbf{Step 2. Estimate of $\calF_N$ and $\mathfrak b_N$.} For $0<\mu<1$, after cancellation of the leading term of $\mu^{-D}(\Lambda W)_\mu-\Lambda W$, we obtain
	\begin{equation}\label{eq:Lambda W diff pointwise}
		\begin{gathered}
			|\mu^{-D}(\Lambda W)_\mu(r)-\Lambda W(r)|\lesssim
			\begin{cases}
				\mu^{-2D},&0<r\leq\mu,\\
				r^{-2D},&\mu\leq r\leq1,\\
				r^{-2D-2},&1\leq r,
			\end{cases}\\
		\end{gathered}
	\end{equation}
	Using this, we obtain
	\begin{equation}\label{eq:calF integrand esti}
		|(s^{-D}(\Lambda W)_s-\Lambda W,s^{-1}(\Lambda W)_s)_d|\lesssim s^{-1/2}.
	\end{equation}
	Thus, we have
	\begin{equation*}
		|\calF_N|\leq\lambda_{N-1}^{2-D}|{\textstyle\int_0^{\mu_N}}(s^{-D}(\Lambda W)_s-\Lambda W,s^{-1}(\Lambda W)_s)_d ds|\lesssim\lambda_{N-1}^{2-D}.
	\end{equation*}
	
	For $\mathfrak b_N$, we decompose it into the contributions on $(0,r_0)$ and $(r_0,\infty)$:
	\begin{equation*}
		\begin{aligned}
			|\mathfrak b_N|
			&\leq\|r^{-1}g\|_{L^2_d(0,r_0)}\|r\Theta_N\|_{L^2_d}
			+|(\Theta_N,\chf_{{r\geq r_0}}g)_d|.
		\end{aligned}
	\end{equation*}
	Here, we have $\|r^{-1}g\|_{L^2_d(0,r_0)}\lesssim \delta^{1/2}$ by \eqref{eq:Hardy Sobolev}.
	After the change of variables $r=\lambda_{N-1}\rho$, the estimate \eqref{eq:Lambda W diff pointwise} with $\mu=\mu_N$ gives
	\begin{equation}\label{eq:Theta pointwise}
		|\Theta_N(r)|\lesssim
		\begin{cases}
			\lambda_N^{-2D},&0<r\leq\lambda_N,\\
			r^{-2D},&\lambda_N\leq r\leq\lambda_{N-1},\\
			\lambda_{N-1}^2r^{-2D-2},&\lambda_{N-1}\leq r.
		\end{cases}
	\end{equation}
	From this, we have $\|r\Theta_N\|_{L^2_d}\lesssim\lambda_{N-1}^{2-D}$. Thus, we have
	\begin{equation}
		\|r^{-1}g\|_{L^2_d(0,r_0)}\|r\Theta_N\|_{L^2_d}\lesssim \delta^{1/2}\lambda_{N-1}^{2-D}.
	\end{equation}
	Using \eqref{eq:radial Linfty}, we get
	\begin{equation}\label{eq:rD g L inf esti}
		\|r^Dg\|_{L^\infty}\lesssim 
		\begin{cases}
			\|v\|_{L^\infty}+\|\sum_j \iota_jQ_{[\lambda_j]}\|_{L^\infty},&\text{for \eqref{eq:HMHF}},\\
			\|g\|_{\dot H^1_d}\leq \|u\|_{\dot H^1_d}+\|U\|_{\dot H^1_d},&\text{for \eqref{eq:NLH}},
		\end{cases}
		\lesssim 1.
	\end{equation}
	Note that, in the \eqref{eq:NLH} case, we used the assumption $\sup_t\|u\|_{\dot H^1_d}<\infty$. By \eqref{eq:Theta pointwise} and \eqref{eq:rD g L inf esti}, we obtain
	\begin{align*}
		|(\Theta_N,\chf_{{r\geq r_0}}g)_d|\lesssim C_{r_0}\lambda_{N-1}^2.
	\end{align*}
	Moreover, when $T_+=\infty$, we have $r_0=\infty$, and hence this term vanishes. When $T_+<\infty$, since $\lambda_{N-1}=o(\sqrt{T_+-t})$, the term $C_{r_0}\lambda_{N-1}^2$ is negligible compared with $\lambda_{N-1}^{2-D}$. Thus, we conclude
	\begin{equation}\label{eq:frak b esti}
		|\mathfrak b_N|\lesssim \lambda_{N-1}^{2-D}+C_{r_0}\lambda_{N-1}^2\lesssim \lambda_{N-1}^{2-D},
	\end{equation}
	which proves \eqref{eq:F b bound}.

	\textbf{Step 3. Estimate of $\calR_{\calF}$.} We first estimate the summation part for $1\leq j\leq N-2$. For such $j$, by \eqref{eq:Lambda W pointwise} and \eqref{eq:Theta pointwise}, we get
	\begin{equation}\label{eq:j ThetaN estimate}
		|(\Theta_N,\lambda_j^{-1}\iota_j(\Lambda W)_{\lambda_j})_d|
		\lesssim\lambda_{N-1}^{1-D}(\tfrac{\lambda_{N-1}}{\lambda_j})^{D+1}
		(1+\log\tfrac{\lambda_j}{\lambda_{N-1}})\lesssim \lambda_{N-1}^{1-D}.
	\end{equation}
	Moreover, we can estimate the first term of $\calR_\calF$ by using \eqref{eq:calF integrand esti}. Therefore, we deduce
	\begin{equation}\label{eq:R lambda bound}
		|\calR_{\calF}|\lesssim\lambda_{N-1}^{1-D}{\textstyle\sum_{j=1}^N}|\lambda_{j,t}|
		\lesssim\lambda_{N-1}^{1-D}\|\calT(u)\|_{L^2_d},
	\end{equation}
	where the last inequality follows from \eqref{eq:lambda t esti}.
	
	\textbf{Step 4. Estimate of $(\Theta_N,-\calT(U))_d$.} Recalling $\kappa$ from \eqref{eq:kappa def}, we prove
	\begin{equation}\label{eq:interaction main}
		(\Theta_N,-\calT(U))_d-\kappa\lambda_{N-1}^{-D}
		=o(\lambda_{N-1}^{-D}).
	\end{equation}
	The main term is determined by $\calT(\iota_{N-1}W_{\lambda_{N-1}}+\iota_NW_{\lambda_N})$, and the remainder becomes an error.
	We first estimate the error terms. Since each $\iota_jW_{\lambda_j}$ is stationary,
	\begin{align*}
		&\calT(U)-\calT(\iota_{N-1}W_{\lambda_{N-1}}+\iota_NW_{\lambda_N})
		\\
		&=\frac1{r^{D+2}}\sum_{j=1}^{N-2}\bigg[f\bigg(\sum_{k=j}^N\iota_kQ_{[\lambda_k]}\bigg)-f(\iota_jQ_{[\lambda_j]})
		-f\bigg(\sum_{k=j+1}^N\iota_kQ_{[\lambda_k]}\bigg)\bigg].
	\end{align*}
	From \eqref{eq:HMHF nonlinear esti}, \eqref{eq:NLH nonlinear esti}, \eqref{eq:profile pointwise}, and \eqref{eq:Lambda W pointwise}, for each $j<k$, we get
	\begin{align*}
		|\calT(U)-\calT(\iota_{N-1}W_{\lambda_{N-1}}+\iota_NW_{\lambda_N})|\lesssim &\chf_{\{r\leq\lambda_k\}}(\lambda_j^{-2}\lambda_k^{-D}+\lambda_j^{-D}\lambda_k^{-2})\\
		&+\chf_{\{\lambda_k<r\leq\lambda_j\}}(\lambda_j^{-2}\lambda_k^Dr^{-2D}+\lambda_j^{-D}\lambda_k^2r^{-4})\\
		&+\chf_{\{r>\lambda_j\}}(\lambda_j^2\lambda_k^D+\lambda_j^D\lambda_k^2)r^{-2D-4}.
	\end{align*}
	By using $|\Theta_N(r)|\lesssim r^{-2D}$ which comes from \eqref{eq:Theta pointwise}, we have
	\begin{equation}\label{eq:outer profile interaction}
		|(\Theta_N,-\calT(U)+\calT(\iota_{N-1}W_{\lambda_{N-1}}+\iota_NW_{\lambda_N}))_d|
		\lesssim {\textstyle\sum_{j=1}^{N-2}}\lambda_j^{-D}
		\lesssim  o(\lambda_{N-1}^{-D}).
	\end{equation}
	When $N=2$, the left-hand side is zero.

    Now, we consider the main contribution. We obtain
    \begin{align*}
    	\calT(\iota_{N-1}W_{\lambda_{N-1}}+\iota_NW_{\lambda_N})=r^{-D-2}f'(\iota_NQ_{[\lambda_N]})\iota_{N-1}Q_{[\lambda_{N-1}]}+\calT_R,
    \end{align*}
    where
    \begin{align*}
    	\calT_R
    	=&r^{-D-2}(f(\iota_NQ_{[\lambda_N]}+\iota_{N-1}Q_{[\lambda_{N-1}]})-f(\iota_NQ_{[\lambda_N]})\\
    	&-f'(\iota_NQ_{[\lambda_N]})\iota_{N-1}Q_{[\lambda_{N-1}]}-f(\iota_{N-1}Q_{[\lambda_{N-1}]})).
    \end{align*}
    Since $f'$ is even, $f'(\iota_NQ_{[\lambda_N]})=f'(Q_{[\lambda_N]})$.
    Using the second and third inequalities of \eqref{eq:HMHF nonlinear esti} and \eqref{eq:NLH nonlinear esti}, together with \eqref{eq:profile pointwise} and \eqref{eq:Lambda W pointwise}, we get
    \begin{align*}
    	|\calT_R|\lesssim
    	\begin{cases}
    		\lambda_{N-1}^{-2D}r^{D-2},&0<r\leq\lambda_N,\\
    		\min\{\lambda_{N-1}^{-2D}r^{D-2},
    		\lambda_N^2\lambda_{N-1}^{-D}r^{-4}
    		+\lambda_N^D\lambda_{N-1}^{-2}r^{-2D}\},&\lambda_N<r\leq\lambda_{N-1},\\
    		\lambda_N^D\lambda_{N-1}^{2-D}r^{-D-4},&\lambda_{N-1}<r,
    	\end{cases}
    \end{align*}
    Combining this bound with \eqref{eq:Theta pointwise}, we obtain
    \begin{align*}
    	|(\Theta_N,\calT_R)_d|\lesssim
    	\mu_N^{D/2}\lambda_{N-1}^{-D}=o(\lambda_{N-1}^{-D}).
    \end{align*}
    
    For the remainder, the main contribution comes from the $\lambda_N^{-D}\iota_{N-1}(\Lambda W)_{\lambda_N}$ term in $\Theta_N$ defined by \eqref{eq:Theta N def}. We first show that the contribution of the $\iota_{N-1}\lambda_{N-1}^{-D}(\Lambda W)_{\lambda_{N-1}}$ term is an error. From \eqref{eq:profile pointwise} and \eqref{eq:Lambda W pointwise}, we get
    \begin{equation*}
    	|r^{-D-2}f'(Q_{[\lambda_N]})\iota_{N-1}Q_{[\lambda_{N-1}]}|\lesssim
    	\begin{cases}
    		\lambda_N^{-2}\lambda_{N-1}^{-D},&0<r\leq\lambda_N,\\
    		\lambda_N^2\lambda_{N-1}^{-D}r^{-4},&\lambda_N<r\leq\lambda_{N-1},\\
    		\lambda_N^2r^{-D-4},&\lambda_{N-1}<r.
    	\end{cases}
    \end{equation*}
    Hence, by this estimate and \eqref{eq:Lambda W pointwise},
    \begin{align*}
    	&|(\iota_{N-1}\lambda_{N-1}^{-D}(\Lambda W)_{\lambda_{N-1}},
    	r^{-D-2}f'(Q_{[\lambda_N]})\iota_{N-1}Q_{[\lambda_{N-1}]})_d|
    	\lesssim \mu_N^D\lambda_{N-1}^{-D}
    	=o(\lambda_{N-1}^{-D}).
    \end{align*}
    It remains to compute the contribution of the $\lambda_N^{-D}\iota_{N-1}(\Lambda W)_{\lambda_N}$ term.
    Since $Q_{[\lambda_{N-1}]}=r^DW_{\lambda_{N-1}}$, we have
    \begin{align*}
    	&-(\lambda_N^{-D}\iota_{N-1}(\Lambda W)_{\lambda_N},
    	r^{-D-2}f'(Q_{[\lambda_N]})\iota_{N-1}Q_{[\lambda_{N-1}]})_d
    	\\
    	&=\kappa\lambda_{N-1}^{-D}
    	-(\lambda_N^{-D}(\Lambda W)_{\lambda_N},r^{-2}f'(Q_{[\lambda_N]})[W_{\lambda_{N-1}}-W_{\lambda_{N-1}}(0)])_d.
    \end{align*}
    The explicit formulas for $W$ imply
    \begin{equation*}
    	|W_{\lambda_{N-1}}(r)-W_{\lambda_{N-1}}(0)|
    	\lesssim\lambda_{N-1}^{-D}\min\{r^2\lambda_{N-1}^{-2},1\}.
    \end{equation*}
    Moreover, \eqref{eq:one bubble potential explicit} and \eqref{eq:Lambda W pointwise} imply
    \begin{equation*}
    	|f'(Q_{[\lambda_N]})|
    	\lesssim
    	\begin{cases}
    		r^2\lambda_N^{-2},&0<r\leq\lambda_N,\\
    		\lambda_N^2r^{-2},&\lambda_N\leq r,
    	\end{cases}
    	\quad
    	\lambda_N^{-D}|(\Lambda W)_{\lambda_N}|
    	\lesssim
    	\begin{cases}
    		\lambda_N^{-2D},&0<r\leq\lambda_N,\\
    		r^{-2D},&\lambda_N\leq r.
    	\end{cases}
    \end{equation*}
    Therefore
    \begin{equation*}
    	|(\lambda_N^{-D}(\Lambda W)_{\lambda_N},
    	r^{-2}f'(Q_{[\lambda_N]})[W_{\lambda_{N-1}}-W_{\lambda_{N-1}}(0)])_d|
    	=o(\lambda_{N-1}^{-D}).
    \end{equation*}
    Combining the above estimates, we arrive at
    \begin{equation*}
    	(\Theta_N,-\calT(\iota_{N-1}W_{\lambda_{N-1}}+\iota_NW_{\lambda_N}))_d
    	=\kappa\lambda_{N-1}^{-D}
    	+o(\lambda_{N-1}^{-D}),
    \end{equation*}
    which concludes \eqref{eq:interaction main}.

	\textbf{Step 5. Estimates of $(H_U\Theta_N,g)_d$.} We prove
	\begin{align}
		|(H_U\Theta_N,g)_d|&\lesssim(\delta^{1/2}+o(1))\lambda_{N-1}^{-D}.\label{eq:g HUTheta estimate}
	\end{align}
	Indeed, we have
	\begin{align*}
		|(H_U\Theta_N,g)_d|&\lesssim \delta^{1/2}\|rH_U\Theta_N\|_{L^2_d}
		+(H_U\Theta_N,\chf_{r>r_0}g)_d.
	\end{align*}
	Thus, proving \eqref{eq:g HUTheta estimate} reduces to showing
	\begin{align}
		\|rH_U\Theta_N\|_{L^2_d}&\lesssim \lambda_{N-1}^{-D}\label{eq:HU ThetaN}
		\\
		(H_U\Theta_N,\chf_{r>r_0}g)_d&\leq C_{r_0}\lambda_{N-1}^2=o(\lambda_{N-1}^{-D}). \label{eq:exterior g HU ThetaN}
	\end{align}
	Using the definition of $H_U$ and $H\Lambda W=0$, we write
	\begin{align*}
		H_U\Theta_N
		=&\iota_{N-1}r^{-2}[f'(Q_{[\lambda_N]})-f'(\iota_{N-1}Q_{[\lambda_{N-1}]}+\iota_NQ_{[\lambda_N]})]
		\lambda_N^{-D}(\Lambda W)_{\lambda_N}\\
		&-\iota_{N-1}r^{-2}[f'(Q_{[\lambda_{N-1}]})-f'(\iota_{N-1}Q_{[\lambda_{N-1}]}+\iota_NQ_{[\lambda_N]})]
		\lambda_{N-1}^{-D}(\Lambda W)_{\lambda_{N-1}}\\
		&+r^{-2}[f'(\iota_{N-1}Q_{[\lambda_{N-1}]}+\iota_NQ_{[\lambda_N]})-f'(P)]\Theta_N,
	\end{align*}
	where $P=\sum_{j=1}^{N}\iota_jQ_{[\lambda_j]}$.
    Using the first inequalities in \eqref{eq:HMHF nonlinear esti} and \eqref{eq:NLH nonlinear esti}, together with \eqref{eq:profile pointwise}, \eqref{eq:Lambda W pointwise}, and $\sin Q_{[\lambda]}=r(\Lambda W)_\lambda$ for \eqref{eq:HMHF}, we obtain, in both models,
    \begin{align}
    	&r^{-2}|[f'(\iota_{N-1}Q_{[\lambda_{N-1}]}+\iota_NQ_{[\lambda_N]})-f'(Q_{[\lambda_N]})]
    	\lambda_N^{-D}(\Lambda W)_{\lambda_N}|\nonumber\\
    	&+r^{-2}|[f'(\iota_{N-1}Q_{[\lambda_{N-1}]}+\iota_NQ_{[\lambda_N]})-f'(Q_{[\lambda_{N-1}]})]
    	\lambda_{N-1}^{-D}(\Lambda W)_{\lambda_{N-1}}| \label{eq:HU ThetaN 1}\\
    	&\lesssim
    	\begin{cases}
    		\lambda_N^{-D-2}\lambda_{N-1}^{-D},&0<r\leq\lambda_N,\\
    		\lambda_N^{2-D}\lambda_{N-1}^{-D}r^{-4}
    		+\lambda_{N-1}^{-2}r^{-2D},&\lambda_N<r\leq\lambda_{N-1},\\
    		\lambda_{N-1}^2r^{-2D-4},&\lambda_{N-1}<r.
    	\end{cases} \nonumber
    \end{align}
    For the last term, we have
    \begin{equation}\label{eq:f prime P expansion}
    	f'(P)-f'(\iota_{N-1}Q_{[\lambda_{N-1}]}+\iota_NQ_{[\lambda_N]})
    	=\sum_{j=1}^{N-2}\bigg[
    	f'\bigg(\sum_{k=j}^N\iota_kQ_{[\lambda_k]}\bigg)
    	-f'\bigg(\sum_{k=j+1}^N\iota_kQ_{[\lambda_k]}\bigg)\bigg].
    \end{equation}
    If $N=2$, then $P=\iota_{N-1}Q_{[\lambda_{N-1}]}+\iota_NQ_{[\lambda_N]}$, so the last term vanishes. Suppose that $N\geq3$. Again, using \eqref{eq:f prime P expansion}, the first inequalities in \eqref{eq:HMHF nonlinear esti} and \eqref{eq:NLH nonlinear esti}, \eqref{eq:profile pointwise}, \eqref{eq:Lambda W pointwise}, \eqref{eq:Theta pointwise}, and $\sin Q_{[\lambda]}=r(\Lambda W)_\lambda$ for \eqref{eq:HMHF}, we obtain directly
    \begin{equation}\label{eq:HU ThetaN 2}
    	\begin{aligned}
    		&r^{-2}|[f'(\iota_{N-1}Q_{[\lambda_{N-1}]}+\iota_NQ_{[\lambda_N]})-f'(P)]\Theta_N|
    		\\
    		&\lesssim 
    		\begin{cases}
    			\lambda_N^{-D-2}\lambda_{N-2}^{-D},&0<r\leq\lambda_N,\\
    			\lambda_N^{2-D}\lambda_{N-2}^{-D}r^{-4}
    			+\lambda_{N-1}^{D-2}\lambda_{N-2}^{-D}r^{-2D},&\lambda_N<r\leq\lambda_{N-1},\\
    			\lambda_{N-1}^2r^{-2D-4},&\lambda_{N-1}<r.
    		\end{cases}
    	\end{aligned}
    \end{equation}
    Therefore, from \eqref{eq:HU ThetaN 1} and \eqref{eq:HU ThetaN 2}, we obtain \eqref{eq:HU ThetaN}.
    
    Next, if $r_0=\infty$, there is nothing to prove for \eqref{eq:exterior g HU ThetaN}; otherwise, by \eqref{eq:rD g L inf esti}, \eqref{eq:HU ThetaN 1}, and \eqref{eq:HU ThetaN 2} with $r_0>\lambda_{N-1}$, we derive \eqref{eq:exterior g HU ThetaN}. Thus, we conclude \eqref{eq:g HUTheta estimate}.
	
	\textbf{Step 6. Estimate of $(\Theta_N,\NL_U(g))_d$.} By \eqref{eq:HMHF nonlinear esti} and \eqref{eq:NLH nonlinear esti} with $|P|\lesssim  1$, we have $|\NL_U(g)|\lesssim r^{D-2}g^2$. Moreover, \eqref{eq:Theta pointwise} implies $|\Theta_N|\lesssim r^{-2D}$. Thus, from \eqref{eq:rD g L inf esti} for $r>r_0$ and \eqref{eq:weighted esti} for $r<r_0$, we have
	\begin{equation}\label{eq:nonlinear bound}
		\begin{aligned}
			|(\Theta_N,\NL_U(g))_d|\lesssim&{\textstyle\int_0^{r_0}}g^2r^{D-1}dr+C_{r_0}\lambda_{N-1}^2
			\\
			\lesssim&\delta^{(2-D)/2}\|\calT(u)\|_{L^2_d}^D
			+(\delta^{1/4}+o(1))\lambda_{N-1}^{-D}.
		\end{aligned}
	\end{equation}
	Here, as in \eqref{eq:frak b esti}, the term $C_{r_0}\lambda_{N-1}^2$ vanishes when $T_+=\infty$ and is absorbed into $o(1)\lambda_{N-1}^{-D}$ otherwise.
    
	\textbf{Step 7. Estimate of $(\partial_t\Theta_N,g)_d$.} First, note that
	\begin{align*}
		|\partial_t\Theta_N|\lesssim
		(|\lambda_{N,t}|+|\lambda_{N-1,t}|)\times
		\begin{cases}
			r^{-2D-1},&0<r\leq\lambda_{N-1},\\
			\lambda_{N-1}r^{-2D-2},&\lambda_{N-1}<r.
		\end{cases}
	\end{align*}
	Since $\lambda_{N-1}<r_0$, from the above estimate and \eqref{eq:rD g L inf esti}, we obtain
	\begin{equation*}
		\begin{aligned}
			|(\partial_t\Theta_N,g)_d|
			\lesssim&
			\bigl[({\textstyle\int_0^{r_0}}g^2r^{D-1}dr)^{1/2}
			\lambda_{N-1}^{1-D/2}+C_{r_0}\lambda_{N-1}\bigr]
			(|\lambda_{N,t}|+|\lambda_{N-1,t}|).
		\end{aligned}
	\end{equation*}
	Here, when $T_+=\infty$, the term $C_{r_0}\lambda_{N-1}$ is absent.
	From \eqref{eq:weighted esti} and \eqref{eq:lambda t esti}, we derive
	\begin{equation}\label{eq:g Theta t}
		\begin{aligned}
			|(\partial_t\Theta_N,g)_d|
			\lesssim&
			\delta^{(2-D)/4}\lambda_{N-1}^{1-D/2}\|\calT(u)\|_{L^2_d}^{1+D/2}
			+(\delta^{1/8}+o(1))\lambda_{N-1}^{1-D}\|\calT(u)\|_{L^2_d}
			\\
			\lesssim &
			\|\calT(u)\|_{L^2_d}^D+\lambda_{N-1}^{2-D}\|\calT(u)\|_{L^2_d}^2+\lambda_{N-1}^{1-D}\|\calT(u)\|_{L^2_d}.
		\end{aligned}
	\end{equation}

	\textbf{Step 8. Completion of the proof.} The error $\calR$ obtained in Step 1 consists of the five terms estimated in \eqref{eq:R lambda bound}, \eqref{eq:interaction main}, \eqref{eq:g HUTheta estimate}, \eqref{eq:nonlinear bound}, and \eqref{eq:g Theta t}. Thus, we have
	\begin{equation*}
		|\calR|\lesssim
		(\delta^{1/4}+o(1))\lambda_{N-1}^{-D}
		+\|\calT(u)\|_{L^2_d}^D
		+\lambda_{N-1}^{2-D}\|\calT(u)\|_{L^2_d}^2+\lambda_{N-1}^{1-D}\|\calT(u)\|_{L^2_d}.
	\end{equation*}
    As in \eqref{eq:frak b esti}, the last term vanishes when $T_+=\infty$ and is negligible otherwise. Hence, we conclude \eqref{eq:mod esti}.
\end{proof}

Now, we finish the proof.

\begin{proof}[Proof of Theorems~\ref{thm:HMHF main} and \ref{thm:NLH main}]
	Suppose that the decomposition contains $N\geq2$ bubbles, and apply Lemma~\ref{lem:modulation}.
	
	We first verify that, when $T_+<\infty$,
	\begin{equation}\label{eq:lambda minus D bdd goal}
		{\textstyle\int_t^{T_+}}\lambda_{N-1}^{-D}ds<\infty.
	\end{equation}
	Recall that the constant $\eta$ is fixed in Lemma~\ref{lem:modulation}, and it satisfies $\eta \ll|\kappa|$. By integrating \eqref{eq:mod esti}, we obtain
	\begin{equation}\label{eq:lambda minus D bdd}
		\begin{aligned}
			(|\kappa|-C\eta){\textstyle\int_t^s}\lambda_{N-1}^{-D}d\sigma
			\lesssim_{\eta}&|\calF_N(s)-\mathfrak b_N(s)|+|\calF_N(t)-\mathfrak b_N(t)|\\
			&+{\textstyle\int_t^s}(\|\calT(u)\|_{L^2_d}^D+\lambda_{N-1}^{2-D}\|\calT(u)\|_{L^2_d}^2)d\sigma.
		\end{aligned}
	\end{equation}
	Here, we used $\lambda_{N-1}^{1-D}\|\calT(u)\|_{L^2_d}\leq\eta\lambda_{N-1}^{-D} +\frac1{4\eta}\lambda_{N-1}^{2-D}\|\calT(u)\|_{L^2_d}^2$. Moreover, we have
	\begin{align}\label{eq:calT D esti}
		{\textstyle\int_t^s}\|\calT(u)\|_{L^2_d}^Dd\sigma \leq(s-t)^{(2-D)/2}
		({\textstyle\int_t^{s}}\|\calT(u)\|_{L^2_d}^2d\sigma)^{D/2}.
	\end{align}
	As $s\to T_+$, \eqref{eq:F b bound} and \eqref{eq:dissipation} keep the right-hand side of \eqref{eq:lambda minus D bdd} bounded. This proves \eqref{eq:lambda minus D bdd goal}.
	
	For $T_+=\infty$, fix $T_0\geq t_1^*$ sufficiently large. Define
	\begin{equation}\label{eq:tau N def}
		\tau_N(t)\coloneqq
		\begin{cases}
			\int_t^{T_+}\lambda_{N-1}^{-D}ds,&T_+<\infty,\\[1mm]
			\int_{T_0}^t\lambda_{N-1}^{-D}ds,&T_+=\infty.
		\end{cases}
	\end{equation}
	By \eqref{eq:lambda minus D bdd goal}, $\tau_N(t)<\infty$ and $\tau_N(t)\to0$ when $T_+<\infty$.
	We claim that
	\begin{equation}\label{eq:tauN growth}
		\frac{\tau_N(t)}{\min\{\sqrt{T_+-t},\sqrt t\}^{2-D}}\to\infty,
		\qquad
		\frac{\tau_N(t)}{\lambda_{N-1}(t)^{2-D}}\to\infty
		\quad\text{as}\quad t\to T_+.
	\end{equation}
	For each $\epsilon>0$, \eqref{eq:modulation smallness} implies
	\begin{equation}\label{eq:pf final lambda N}
		\lambda_{N-1}(s)\leq\lambda_1(s)\leq \epsilon \min\{\sqrt{T_+-s},\sqrt s\}
	\end{equation}
	for all $s$ sufficiently close to $T_+$. Consequently, for some $T_\epsilon>T_0$,
	\begin{equation*}
		\tau_N(t)\geq
		\begin{cases}
			\epsilon^{-D}\int_t^{T_+}(T_+-s)^{-D/2}ds,&T_+<\infty,\\[1mm]
			\epsilon^{-D}\int_{T_\epsilon}^t s^{-D/2}ds,&T_+=\infty.
		\end{cases}
	\end{equation*}
	Since $0<D<2$ and $\epsilon$ is arbitrary, the first limit in \eqref{eq:tauN growth} follows. The second follows from the same pointwise bound for $\lambda_{N-1}$.
	
	Now, again integrating \eqref{eq:mod esti}, we obtain
	\begin{align*}
		&|(\calF_N(s)-\mathfrak b_N(s))-(\calF_N(t)-\mathfrak b_N(t))-\kappa{\textstyle\int_t^s}\lambda_{N-1}^{-D}d\sigma|
		\\
		&\lesssim
		\eta {\textstyle\int_t^s}\lambda_{N-1}^{-D}d\sigma +\eta^{-1}{\textstyle\int_t^s}(\|\calT(u)\|_{L^2_d}^D+\lambda_{N-1}^{2-D}\|\calT(u)\|_{L^2_d}^2)d\sigma.
	\end{align*}
	When $T_+<\infty$, letting $s\to T_+$ in the preceding estimate and using \eqref{eq:F b bound}, \eqref{eq:calT D esti}, \eqref{eq:tauN growth}, and \eqref{eq:pf final lambda N}, we obtain
	\begin{align}\label{eq:mod esti finite T}
		|(\calF_N(t)-\mathfrak b_N(t)) +\kappa\tau_N(t)|
		\lesssim
		\eta \tau_N(t) +o_{t\to T_+}(1)\cdot \tau_N(t).
	\end{align}
	When $T_+=\infty$, \eqref{eq:tauN growth} implies $\tau_N(t)\to\infty$, and hence $\calF_N(T_0)-\mathfrak b_N(T_0)=o_{t}(1)\tau_N(t)$. Applying the preceding estimate on $[T_0,t]$ and using \eqref{eq:calT D esti}, \eqref{eq:tauN growth}, and \eqref{eq:pf final lambda N}, we obtain
	\begin{equation}\label{eq:mod esti inf T}
		|(\calF_N(t)-\mathfrak b_N(t))-\kappa\tau_N(t)|
		\lesssim\eta\tau_N(t)+o_{t\to\infty}(1)\tau_N(t).
	\end{equation}
	Since $\eta\ll 1$, \eqref{eq:mod esti finite T} and \eqref{eq:mod esti inf T} imply
	\begin{equation*}
		\liminf_{t\to T_+}\frac{|\calF_N(t)-\mathfrak b_N(t)|}{\tau_N(t)}>0.
	\end{equation*}
	On the other hand, \eqref{eq:F b bound} and \eqref{eq:tauN growth} imply
	\begin{equation*}
		\frac{|\calF_N(t)-\mathfrak b_N(t)|}{\tau_N(t)}
		\lesssim\frac{\lambda_{N-1}(t)^{2-D}}{\tau_N(t)}\to0.
	\end{equation*}
	This contradiction excludes $N\geq2$. Therefore, the decompositions in Propositions~\ref{prop:sol resol hmhf} and~\ref{prop:sol resol nlh} contain at most one bubble, which yields the conclusions of Theorems~\ref{thm:HMHF main} and~\ref{thm:NLH main}.
\end{proof}

\section{Proof of Lemma~\ref{lem:nonlinear esti}}\label{sec:technical lem}

In this section, we prove Lemma~\ref{lem:nonlinear esti}, which supplies the weighted control of $g$ required in the proof of Proposition~\ref{prop:modulation estimates}. The proof is based on the coercivity estimate \eqref{eq:H coercivity}.

We treat the cases $D\in\{1,\frac32\}$ and $D=\frac12$ separately. The former includes the $1$-equivariant \eqref{eq:HMHF} and \eqref{eq:NLH} in dimensions $d=4,5$, while the latter corresponds to \eqref{eq:NLH} in dimension $d=3$. When $D\in\{1,\frac32\}$, we apply \eqref{eq:H coercivity} directly to a localization of $g$ around the bubble at $\lambda_N$. When $D=\frac12$, the slow decay of $W_{\lambda_N}$ makes its interaction with the bubbles at $\lambda_1,\ldots,\lambda_{N-1}$ too large for the same argument. We therefore introduce a modified profile.

We first consider the cases $D\in\{1,\frac32\}$.

\begin{lem}[Local coercivity estimate: $D\in\{1,\frac32\}$]
	Let $D\in\{1,\frac32\}$, and fix a sufficiently small constant $\delta_0>0$. For every $t\in[t_1^*,T_+)$ sufficiently close to $T_+$, we have
	\begin{equation}\label{eq:H2 esti for D one three halves}
		\|\chi_{\delta_0\lambda_{N-1}}g\|_{\dot H^2_d}
		\lesssim\|\calT(u)\|_{L^2_d}+\lambda_{N-1}^{-1}.
	\end{equation}
\end{lem}

\begin{proof}
	\textbf{Step 1.} Recall $P=\sum_{j=1}^{N}\iota_jQ_{[\lambda_j]}$. We first claim
	\begin{align}
		\|\chf_{\{r<2\delta_0\lambda_{N-1}\}}r^{-2}(f'(P)-f'(Q_{[\lambda_N]}))g\|_{L^2_d}
		&\lesssim\lambda_{N-1}^{-1}\|\chi_{2r_0}g\|_{\dot H^1_d},\label{eq:potential esti}\\
		\|\calT(U)\|_{L^2_d}=\|\Delta_dU+r^{-(D+2)}f(r^DU)\|_{L^2_d}
		&\lesssim \lambda_{N-1}^{-1}.\label{eq:calT U esti}
	\end{align}
	Recall that $f'$ is even, and hence $f'(\iota_NQ_{[\lambda_N]})=f'(Q_{[\lambda_N]})$. Moreover, we have
	\begin{equation}\label{eq:inner profile esti}
		|P-\iota_NQ_{[\lambda_N]}|\lesssim {\textstyle\sum_{j=1}^{N-1}}(\lambda_j^{-1}r)^D\lesssim (\lambda_{N-1}^{-1}r)^D,\quad 0<r<2\delta_0\lambda_{N-1}.
	\end{equation}
	Moreover,
	\begin{equation}\label{eq:profile nonlinear esti}
		\begin{cases}
			|\sin Q_{[\lambda_N]}(r)|,&\text{for \eqref{eq:HMHF}},\\
			|Q_{[\lambda_N]}(r)|^{2/D-1},&\text{for \eqref{eq:NLH}}
		\end{cases}
		\lesssim
		\begin{cases}
			(\lambda_N^{-1}r)^{2-D},&0<r\leq\lambda_N,\\
			(\lambda_Nr^{-1})^{2-D},&\lambda_N\leq r.
		\end{cases}
	\end{equation}
	Indeed, the \eqref{eq:HMHF} estimate follows from $\sin Q_{[\lambda_N]}=r(\Lambda W)_{\lambda_N}$ and \eqref{eq:Lambda W pointwise}, while the \eqref{eq:NLH} estimate follows from \eqref{eq:profile pointwise}.
	
	With $a=\iota_NQ_{[\lambda_N]}$ and $b=P-\iota_NQ_{[\lambda_N]}$, we apply \eqref{eq:HMHF nonlinear esti} and \eqref{eq:NLH nonlinear esti} in the \eqref{eq:HMHF} and \eqref{eq:NLH} cases, respectively. Combining the resulting bounds with \eqref{eq:inner profile esti} and \eqref{eq:profile nonlinear esti}, we obtain
	\begin{equation}\label{eq:potential comparison pointwise}
		\begin{aligned}
			|f'(P)-f'(Q_{[\lambda_N]})|
			&\lesssim(\lambda_{N-1}^{-1}r)^2
			+(\lambda_{N-1}^{-1}r)^D(\lambda_N^{-1}r)^{2-D}\chf_{\{r\leq\lambda_N\}}\\
			&\quad+(\lambda_{N-1}^{-1}r)^D(\lambda_Nr^{-1})^{2-D}
			\chf_{\{\lambda_N<r<2\delta_0\lambda_{N-1}\}}.
		\end{aligned}
	\end{equation}
	For $t$ sufficiently close to $T_+$, we have $2\delta_0\lambda_{N-1}<r_0$. Hence, \eqref{eq:radial Linfty} and \eqref{eq:potential comparison pointwise}, with $d=2D+2$ and $\delta_0, \mu_N\lesssim 1$, imply
	\begin{align*}
		\|\chf_{\{r<2\delta_0\lambda_{N-1}\}}r^{-2}(f'(P)-f'(Q_{[\lambda_N]}))g\|_{L^2_d}^2
		\lesssim \lambda_{N-1}^{-2}\|\chi_{2r_0}g\|_{\dot H^1_d}^2,
	\end{align*}
	which proves \eqref{eq:potential esti}.
	
	We next prove \eqref{eq:calT U esti}. Since each $\iota_jW_{\lambda_j}$ is stationary,
	\begin{equation}\label{eq:U interaction identity}
		\calT(U)=r^{-(D+2)}[f(P)-{\textstyle\sum_{j=1}^N}f(\iota_jQ_{[\lambda_j]})].
	\end{equation}
	For \eqref{eq:HMHF} and \eqref{eq:NLH}, respectively, we combine \eqref{eq:HMHF nonlinear esti} with $\sin Q_{[\lambda]}=r(\Lambda W)_\lambda$ and \eqref{eq:NLH nonlinear esti} with $Q_{[\lambda]}=r^DW_\lambda$ to obtain
	\begin{equation*}
		|\calT(U)|\lesssim \sum_{1\leq j<k\leq N}
		\begin{cases}
			|(\Lambda W)_{\lambda_j}||(\Lambda W)_{\lambda_k}|[|(\Lambda W)_{\lambda_j}|+|(\Lambda W)_{\lambda_k}|],&\text{for \eqref{eq:HMHF}},\\
			|W_{\lambda_j}|^{2/D}|W_{\lambda_k}|+|W_{\lambda_k}|^{2/D}|W_{\lambda_j}|,&\text{for \eqref{eq:NLH}}.
		\end{cases}
	\end{equation*}
	Applying \eqref{eq:Lambda W pointwise} for \eqref{eq:HMHF} and \eqref{eq:profile pointwise} for \eqref{eq:NLH}, and estimating separately on the regions $0<r<\lambda_k$, $\lambda_k<r<\lambda_j$, and $\lambda_j<r$, we obtain
	\begin{align*}
		\|\calT(\iota_jW_{\lambda_j}+\iota_kW_{\lambda_k})\|_{L^2_d}^2
		\lesssim(\lambda_k/\lambda_j)^{2D-2}\lambda_j^{-2}.
	\end{align*}
	The last estimates in \eqref{eq:HMHF nonlinear esti} and \eqref{eq:NLH nonlinear esti} therefore yield
	\begin{equation*}
		\|\calT(U)\|_{L^2_d}
		\lesssim \sum_{1\leq j<k\leq N}(\lambda_k/\lambda_j)^{D-1}\lambda_j^{-1}
		\lesssim \lambda_{N-1}^{-1}.
	\end{equation*}
	This proves \eqref{eq:calT U esti}.
	
	\textbf{Step 2.} We prove \eqref{eq:H2 esti for D one three halves}. Since $\chi_{\delta_0\lambda_{N-1}}=1$ on $\operatorname{supp}Z_{\underline{\lambda_N}}$, we have
	\begin{equation*}
		(Z_{\underline{\lambda_N}},\chi_{\delta_0\lambda_{N-1}}g)_d=0.
	\end{equation*}
	Moreover, we have
	\begin{equation*}
		H_{\lambda_N}(\chi_{\delta_0\lambda_{N-1}}g)
		=\chi_{\delta_0\lambda_{N-1}}[H_Ug+r^{-2}\{f'(P)-f'(Q_{[\lambda_N]})\}g]
		+[-\Delta_d,\chi_{\delta_0\lambda_{N-1}}]g.
	\end{equation*}
	From the identity $[-\Delta_d,\chi_{\delta_0\lambda_{N-1}}]g=-2(\chi_{\delta_0\lambda_{N-1}})_rg_r-(\Delta_d\chi_{\delta_0\lambda_{N-1}})g$ and \eqref{eq:potential esti},
	\begin{equation*}
		\|[-\Delta_d,\chi_{\delta_0\lambda_{N-1}}]g\|_{L^2_d}
		+\|\chi_{\delta_0\lambda_{N-1}}r^{-2}[f'(P)-f'(Q_{[\lambda_N]})]g\|_{L^2_d}
		\lesssim\lambda_{N-1}^{-1}\|\chi_{2r_0}g\|_{\dot H^1_d}.
	\end{equation*}
	Hence \eqref{eq:H coercivity} implies
	\begin{equation}\label{eq:coercivity high dim}
		\|\chi_{\delta_0\lambda_{N-1}}g\|_{\dot H^2_d}
		\lesssim\|H_Ug\|_{L^2_d(r<2\delta_0\lambda_{N-1})}
		+\lambda_{N-1}^{-1}\|\chi_{2r_0}g\|_{\dot H^1_d}.
	\end{equation}
	To estimate the $H_Ug$ term, we recall \eqref{eq:error equ},
	\begin{equation}\label{eq:error equ section 4}
		H_Ug=\calT(U)-\calT(u)+\NL_U(g).
	\end{equation}
	Combining \eqref{eq:radial Linfty}, \eqref{eq:HMHF nonlinear esti}, and \eqref{eq:NLH nonlinear esti}, we get
	\begin{equation}\label{eq:quadratic nonlinearity}
		|\NL_U(g)|\lesssim r^{D-2}g^2,\qquad 0<r<r_0.
	\end{equation}
	On $r<\delta_0\lambda_{N-1}$, \eqref{eq:first order L4} and \eqref{eq:coercivity high dim} yield
	\begin{equation*}
		\|\NL_U(g)\|_{L^2_d(r<\delta_0\lambda_{N-1})}
		\lesssim\|\chi_{2r_0}g\|_{\dot H^1_d}\|H_Ug\|_{L^2_d(r<2\delta_0\lambda_{N-1})}
		+\lambda_{N-1}^{-1}\|\chi_{2r_0}g\|_{\dot H^1_d}^2.
	\end{equation*}
	On $\delta_0\lambda_{N-1}<r<2\delta_0\lambda_{N-1}$, \eqref{eq:radial Linfty} and \eqref{eq:Hardy Sobolev} give
	\begin{equation*}
		\|\chf_{\{\delta_0\lambda_{N-1}<r<2\delta_0\lambda_{N-1}\}}\NL_U(g)\|_{L^2_d}
		\lesssim\lambda_{N-1}^{-1}\|\chi_{2r_0}g\|_{\dot H^1_d}^2.
	\end{equation*}
	Therefore
	\begin{equation}\label{eq:inner nonlinear esti}
		\|\NL_U(g)\|_{L^2_d(r<2\delta_0\lambda_{N-1})}
		\lesssim\|\chi_{2r_0}g\|_{\dot H^1_d}\|H_Ug\|_{L^2_d(r<2\delta_0\lambda_{N-1})}
		+\lambda_{N-1}^{-1}\|\chi_{2r_0}g\|_{\dot H^1_d}^2.
	\end{equation}
	Using \eqref{eq:error equ section 4}, \eqref{eq:calT U esti}, and \eqref{eq:inner nonlinear esti}, we obtain
	\begin{equation*}
		\|H_Ug\|_{L^2_d(r<2\delta_0\lambda_{N-1})}
		\lesssim\|\calT(u)\|_{L^2_d}+\lambda_{N-1}^{-1}
		+\|\chi_{2r_0}g\|_{\dot H^1_d}\|H_Ug\|_{L^2_d(r<2\delta_0\lambda_{N-1})}.
	\end{equation*}
	Using \eqref{eq:delta bound} to absorb the last term, we derive
	\begin{equation}\label{eq:HU local esti}
		\|H_Ug\|_{L^2_d(r<2\delta_0\lambda_{N-1})}
		\lesssim\|\calT(u)\|_{L^2_d}+\lambda_{N-1}^{-1}.
	\end{equation}
	Combining \eqref{eq:coercivity high dim} and \eqref{eq:HU local esti}, we conclude \eqref{eq:H2 esti for D one three halves}.
\end{proof}

We next prove a lemma for the case $D=1/2$.

\begin{lem}[Modified profile, $D=1/2$]\label{lem:modified profile}
	We consider \eqref{eq:NLH} with $D=\frac12$ ($d=3$). For each $t\in[t_1^*,T_+)$ sufficiently close to $T_+$, there exists $\nu_N(t)>0$ with the following properties. We have
	\begin{equation}\label{eq:nuN correction}
		|\log(\nu_N/\lambda_N)|\lesssim\mu_N^D.
	\end{equation}
	For a modified profile
	\begin{equation}\label{eq:modified profile}
		\td U(r)\coloneqq\iota_N\left(W_{\nu_N}(r)-W_{\lambda_N}(r)\right)
		-\bigg(\sum_{j=1}^{N-1}\iota_j\lambda_j^{-D}\bigg)\chi_{4\delta\lambda_{N-1}}(r),
	\end{equation}
	we have the modified decomposition
	\begin{equation}\label{eq:modified error}
		w\coloneqq u-(U+\td U)=g-\td U,\qquad (Z_{\underline{\nu_N}},w)_3=0.
	\end{equation}
	Moreover, we have
	\begin{equation}\label{eq:H1 esti for three dim}
		\|\chi_{2r_0}w\|_{\dot H^1_3}\lesssim\delta^{1/2},
	\end{equation}
	and
	\begin{equation}\label{eq:H2 esti for three dim}
		\|\chi_{\frac{r_0}{4}}w\|_{\dot H^2_3}
		\lesssim\|\calT(u)\|_{L^2_3}+\lambda_{N-1}^{-1}\delta^{-1/2}.
	\end{equation}
\end{lem}

\begin{proof}
	\textbf{Step 1.} We first show \eqref{eq:nuN correction}--\eqref{eq:H1 esti for three dim}.
	After increasing $t_1^*$ if necessary, we may assume $R_0\lambda_1<\frac{r_0}{16}$ and $8\delta\lambda_{N-1}<\frac{r_0}{8}$.
	Fix a sufficiently small $c_0>0$. For $|s|\leq c_0$, define
	\begin{equation*}
		\mathbf F(s)\coloneqq (\iota_NZ_{\underline{\lambda_Ne^s}},
		g+\iota_N(W_{\lambda_N}-W_{\lambda_Ne^s})
		+({\textstyle\sum_{j=1}^{N-1}}\iota_j\lambda_j^{-D})\chi_{4\delta\lambda_{N-1}})_3.
	\end{equation*}
	Since $(Z_{\underline{\lambda_N}},g)_3=0$,
	\begin{equation}\label{eq:modified F zero}
		|\mathbf F(0)|=|(\iota_NZ_{\underline{\lambda_N}},
		({\textstyle\sum_{j=1}^{N-1}}\iota_j\lambda_j^{-D})\chi_{4\delta\lambda_{N-1}})_3|
		\lesssim\lambda_N^D{\textstyle\sum_{j=1}^{N-1}}\lambda_j^{-D}
		\lesssim \mu_N^D.
	\end{equation}
	Using
	\begin{equation*}
		\partial_sW_{\lambda_Ne^s}=-(\Lambda W)_{\lambda_Ne^s}, \qquad
		\partial_sZ_{\underline{\lambda_Ne^s}}=-(\Lambda_{-1}Z)_{\underline{\lambda_Ne^s}},
	\end{equation*}
	and $(Z_{\underline{\lambda_Ne^s}},(\Lambda W)_{\lambda_Ne^s})_3=(Z,\Lambda W)_3=1$, we obtain
	\begin{equation*}
		\begin{aligned}
			\partial_s\mathbf F(s)=1-(\iota_N(\Lambda_{-1}Z)_{\underline{\lambda_Ne^s}},
			g+\iota_N(W_{\lambda_N}-W_{\lambda_Ne^s})
			+({\textstyle\sum_{j=1}^{N-1}}\iota_j\lambda_j^{-D})\chi_{4\delta\lambda_{N-1}})_3.
		\end{aligned}
	\end{equation*}
	For the $g$ term, we have
	\begin{align*}
		|(\iota_N(\Lambda_{-1}Z)_{\underline{\lambda_Ne^s}},g)_3|\lesssim\|\chi_{2r_0}g\|_{\dot H^1_3}
		\lesssim\delta^{1/2}.
	\end{align*}
	For the second term, using
	\begin{equation}\label{eq:Lambda W integrate}
		W_{\lambda_N}-W_{\lambda_Ne^s}
		={\textstyle\int_0^{s}}(\Lambda W)_{\lambda_Ne^\sigma}d\sigma,
	\end{equation}
	we arrive at
	\begin{equation*}
		|((\Lambda_{-1}Z)_{\underline{\lambda_Ne^s}},W_{\lambda_N}-W_{\lambda_Ne^s})_3|
		\lesssim |s|.
	\end{equation*}
	For the last term, using the support of $Z$ with $4\delta \lambda_{N-1} \gg \lambda_N e^s$, we obtain
	\begin{equation*}
		(\iota_N(\Lambda_{-1}Z)_{\underline{\lambda_Ne^s}},
		({\textstyle\sum_{j=1}^{N-1}}\iota_j\lambda_j^{-D})\chi_{4\delta\lambda_{N-1}})_3
		\lesssim ({\textstyle\sum_{j=1}^{N-1}}\lambda_j^{-D})(\lambda_{N}e^s)^D\lesssim  \mu_N^D.
	\end{equation*}
	Thus, we get
	\begin{equation}\label{eq:modified F derivative}
		|\partial_s\mathbf F(s)-1|\lesssim\delta^{1/2}+|s|+\mu_N^D.
	\end{equation}
	Choose $c_0$ sufficiently small and increase $t_1^*$ so that the right-hand side of \eqref{eq:modified F derivative} is less than $\frac12$. The implicit function theorem and \eqref{eq:modified F zero} yield a unique $s_N\in(-c_0,c_0)$ satisfying
	\begin{equation*}
		\mathbf F(s_N)=0, \qquad |s_N|\lesssim\mu_N^D.
	\end{equation*}
	Define
	\begin{equation*}
		\nu_N\coloneqq\lambda_Ne^{s_N}.
	\end{equation*}
	Then, we obtain \eqref{eq:nuN correction}, and $\mathbf F(s_N)=0$ is exactly the same as \eqref{eq:modified error}. 
	
	Next, we show \eqref{eq:H1 esti for three dim}. The cutoff correction in $\mathbf F$ satisfies
	\begin{equation*}
		\|({\textstyle\sum_{j=1}^{N-1}}\iota_j\lambda_j^{-D})\chi_{4\delta\lambda_{N-1}}\|_{\dot H^1_3}
		\lesssim({\textstyle\sum_{j=1}^{N-1}}\lambda_j^{-D})(\delta\lambda_{N-1})^D
		\lesssim \delta^{1/2}.
	\end{equation*}
	Moreover, from \eqref{eq:Lambda W integrate} and \eqref{eq:nuN correction}, we derive
	\begin{equation*}
		\|W_{\nu_N}-W_{\lambda_N}\|_{\dot H^1_3} \lesssim \mu_N^D.
	\end{equation*}
	Therefore, from \eqref{eq:modified profile}, we have
	\begin{equation*}
		\|\td U\|_{\dot H^1_3}\lesssim\delta^{1/2}.
	\end{equation*}
	Since $\chi_{2r_0}w=\chi_{2r_0}(g-\td U)$, the definition of $\delta$ establishes \eqref{eq:H1 esti for three dim}.
	
	\textbf{Step 2.} Before proving \eqref{eq:H2 esti for three dim}, we claim the following estimates for the modified profile:
	\begin{equation}\label{eq:modified profile esti}
		\begin{gathered}
			\|U+\td U\|_{L^6_3}\lesssim  1,
			\qquad
			\|\calT(U+\td U)\|_{L^2_3}
			\lesssim \lambda_{N-1}^{-1}\delta^{-1/2},\\
			\||U+\td U|^4-W_{\nu_N}^4\|_{L^3_3}
			\lesssim \lambda_{N-1}^{-1}.
		\end{gathered}
	\end{equation}
	By \eqref{eq:nuN correction}, we have
	\begin{equation*}
		\nu_N\sim\lambda_N, \qquad
		\delta\lambda_{N-1}/\nu_N\sim \delta/\mu_N \geq\mu_N^{-1/2}\to\infty, \qquad
		\delta\lambda_{N-1}/\lambda_j\leq\delta.
	\end{equation*}
	From the explicit formula for the profile $W_\lambda$, we have
	\begin{equation}\label{eq:modified profile inner pointwise}
		|U+\td U-\iota_NW_{\nu_N}|
		\lesssim r^2\lambda_{N-1}^{-5/2},
		\quad \text{on}\quad 0<r<4\delta\lambda_{N-1},
	\end{equation}
	and
	\begin{equation}\label{eq:modified profile esti 2}
		\|U+\td U-\iota_NW_{\nu_N}\|_{L^\infty}
		\lesssim \lambda_{N-1}^{-1/2}, \qquad
		\|U+\td U-\iota_NW_{\nu_N}\|_{L^6_3}\lesssim 1.
	\end{equation}
	Here, we have
	\begin{equation*}
		\|U+\td U\|_{L^6_3}\lesssim \|U+\td U-\iota_NW_{\nu_N}\|_{L^6_3}+\|W_{\nu_N}\|_{L^6_3}\lesssim  1.
	\end{equation*}
	
	Now we estimate $\calT(U+\td U)$.
	The cutoff correction also satisfies
	\begin{equation}\label{eq:modified cutoff esti}
		\| \td U-\iota_N(W_{\nu_N}-W_{\lambda_N}) \|_{\dot H^2_3}
		=
		\| ({\textstyle\sum_{j=1}^{N-1}}\iota_j\lambda_j^{-D})\chi_{4\delta\lambda_{N-1}} \|_{\dot H^2_3}
		\lesssim  \lambda_{N-1}^{-1}\delta^{-1/2}.
	\end{equation}
	Subtracting the stationary equations, we obtain
	\begin{equation}\label{eq:modified nonlinear esti calT 1}
		\begin{aligned}
			\calT(U+\td U)
			=&\Delta_3(\td U-\iota_N(W_{\nu_N}-W_{\lambda_N}))
			-{\textstyle\sum_{j=1}^{N-1}}\iota_jW_{\lambda_j}^5\\
			&+|U+\td U|^4(U+\td U)-\iota_NW_{\nu_N}^5.
		\end{aligned}
	\end{equation}
	Using \eqref{eq:NLH nonlinear esti} with $a=\iota_NW_{\nu_N}$ and $b=U+\td U-\iota_NW_{\nu_N}$, we have
	\begin{equation}\label{eq:modified nonlinear esti calT 2}
		||U+\td U|^4(U+\td U)-\iota_NW_{\nu_N}^5|
		\lesssim W_{\nu_N}^4|U+\td U-\iota_NW_{\nu_N}|
		+|U+\td U-\iota_NW_{\nu_N}|^5.
	\end{equation}
	Using \eqref{eq:modified profile inner pointwise} and \eqref{eq:modified profile esti 2} on $r<4\delta\lambda_{N-1}$ and $4\delta\lambda_{N-1}<r$, respectively, we estimate separately over the regions $0<r<\nu_N$, $\nu_N<r<4\delta\lambda_{N-1}$, and $4\delta\lambda_{N-1}<r$ to obtain
	\begin{equation}\label{eq:modified nonlinear esti calT 3}
		\|W_{\nu_N}^4(U+\td U-\iota_NW_{\nu_N})\|_{L^2_3}^2
		\lesssim\lambda_{N-1}^{-2}(\mu_N^3+\mu_N^4\delta^{-5})
		\lesssim\lambda_{N-1}^{-2}.
	\end{equation}
	On the other hand, the estimates \eqref{eq:modified profile esti 2} imply
	\begin{equation}\label{eq:modified nonlinear esti calT 4}
		\|(U+\td U-\iota_NW_{\nu_N})^5\|_{L^2_3}
		\lesssim \lambda_{N-1}^{-1}.
	\end{equation}
	Since $\sum_{j=1}^{N-1}\|W_{\lambda_j}^5\|_{L^2_3}\lesssim \lambda_{N-1}^{-1}$,
	the estimates \eqref{eq:modified cutoff esti}--\eqref{eq:modified nonlinear esti calT 4} establish
	\begin{equation*}
		\|\calT(U+\td U)\|_{L^2_3}\lesssim \lambda_{N-1}^{-1}\delta^{-1/2}.
	\end{equation*}
	
	For the last inequality in \eqref{eq:modified profile esti}, we apply the first estimate in \eqref{eq:NLH nonlinear esti} with $a=\iota_NW_{\nu_N}$ and $b=U+\td U-\iota_NW_{\nu_N}$:
	\begin{equation}\label{eq:modified potential 1}
		||U+\td U|^4-W_{\nu_N}^4|\lesssim W_{\nu_N}^3|U+\td U-\iota_NW_{\nu_N}|+|U+\td U-\iota_NW_{\nu_N}|^4.
	\end{equation}
	Arguing as in the proof of \eqref{eq:modified nonlinear esti calT 3}, we obtain
	\begin{equation}\label{eq:modified potential 2}
		\|W_{\nu_N}^3(U+\td U-\iota_NW_{\nu_N})\|_{L^3_3}^3
		\lesssim \lambda_{N-1}^{-3}
		\mu_N^{9/2}(1+\log(\delta\lambda_{N-1}\nu_N^{-1})+\delta^{-6})
		\lesssim\lambda_{N-1}^{-3}.
	\end{equation}
	Moreover, from \eqref{eq:modified profile esti 2}, we have
	\begin{equation}\label{eq:modified potential 3}
		\|U+\td U-\iota_NW_{\nu_N}\|_{L^{12}_3} \lesssim \lambda_{N-1}^{-1/4}.
	\end{equation}
	Combining \eqref{eq:modified potential 1}--\eqref{eq:modified potential 3}, we arrive at
	\begin{equation*}
		\||U+\td U|^4-W_{\nu_N}^4\|_{L^3_3}
		\lesssim \lambda_{N-1}^{-1},
	\end{equation*}
	which completes the proof of \eqref{eq:modified profile esti}.
	
	\textbf{Step 3.} Expanding $\calT(U+\td U+w)=\calT(u)$, we obtain
	\begin{equation}\label{eq:cutoff w equ}
		\begin{aligned}
			H_{\nu_N}(\chi_{\frac{r_0}{4}}w)
			&=-\chi_{\frac{r_0}{4}}\calT(u)+\chi_{\frac{r_0}{4}}\calT(U+\td U)
			+5\chi_{\frac{r_0}{4}}(|U+\td U|^4-W_{\nu_N}^4)w\\
			&\quad+[-\Delta_3,\chi_{\frac{r_0}{4}}]w+\chi_{\frac{r_0}{4}}\NL_{U+\td U}(w),
		\end{aligned}
	\end{equation}
	where $H_{\nu_N}=-\Delta_3-5W_{\nu_N}^4$. From \eqref{eq:H1 esti for three dim}, \eqref{eq:Hardy Sobolev}, and \eqref{eq:modified profile esti},
	\begin{equation}\label{eq:d three commutator}
		\begin{aligned}
			&\|\chi_{\frac{r_0}{4}}\calT(U+\td U)\|_{L^2_3}
			+\|\chi_{\frac{r_0}{4}}(|U+\td U|^4-W_{\nu_N}^4)w\|_{L^2_3}
			+\|[-\Delta_3,\chi_{\frac{r_0}{4}}]w\|_{L^2_3}\\
			&\lesssim\lambda_{N-1}^{-1}\delta^{-1/2}
			+\lambda_{N-1}^{-1}\delta^{1/2}+C_{r_0}\delta^{1/2}
			\lesssim\lambda_{N-1}^{-1}\delta^{-1/2}+C_{r_0}\delta^{1/2}.
		\end{aligned}
	\end{equation}
	
	It remains to estimate the nonlinear remainder. 
	By \eqref{eq:NLH nonlinear esti},
	\begin{equation*}
		|\NL_{U+\td U}(w)|\lesssim|U+\td U|^3|w|^2+|w|^5.
	\end{equation*}
	On $r<r_0/8$, \eqref{eq:radial L inf} implies
	\begin{equation}\label{eq:L inf w radial}
		\|\chi_{\frac{r_0}{8}}w\|_{L^\infty}^2
		\lesssim\|\chi_{\frac{r_0}{4}}w\|_{\dot H^1_3}\|\chi_{\frac{r_0}{4}}w\|_{\dot H^2_3}
		\lesssim\delta^{1/2}\|\chi_{\frac{r_0}{4}}w\|_{\dot H^2_3}.
	\end{equation}
	Hence \eqref{eq:L inf w radial}, \eqref{eq:H1 esti for three dim}, and \eqref{eq:modified profile esti} imply
	\begin{equation}\label{eq:d three inner nonlinear}
		\|\chi_{\frac{r_0}{8}}\NL_{U+\td U}(w)\|_{L^2_3}
		\lesssim (\|U+\td U\|_{L^6_3}^3+\|\chi_{\frac{r_0}{4}}w\|_{L^6_3}^3) \|\chi_{\frac{r_0}{8}}w\|_{L^\infty}^2
		\lesssim \delta^{1/2}\|\chi_{\frac{r_0}{4}}w\|_{\dot H^2_3}.
	\end{equation}
	On $r_0/8<r<r_0/2$, we have
	\begin{gather*}
		|w(r)|=|\chi_{2r_0}w(r)|\lesssim r^{-1/2}\|\partial_r(\chi_{2r_0}w)\|_{L^2_3(r,\infty)}
		\lesssim C_{r_0}\|\chi_{2r_0}w\|_{\dot H^1_3}
		\lesssim C_{r_0}\delta^{1/2},
		\\
		|U(r)+\td U(r)|\lesssim  C_{r_0}.
	\end{gather*}
	By these estimates and \eqref{eq:H1 esti for three dim}, we get
	\begin{equation}\label{eq:d three cutoff nonlinear}
		\|\chf_{\{\frac{r_0}{8}<r<\frac{r_0}{2}\}}\chi_{\frac{r_0}{4}}\NL_{U+\td U}(w)\|_{L^2_3}
		\lesssim C_{r_0}\delta^{1/2}.
	\end{equation}
	Note that, when $r_0=\infty$, the last term of \eqref{eq:d three commutator} and \eqref{eq:d three cutoff nonlinear} vanish.
	
	Since $\operatorname{supp}Z_{\underline{\nu_N}}\subset\{r<r_0/4\}$, \eqref{eq:modified error} implies $(\iota_NZ_{\underline{\nu_N}},\chi_{\frac{r_0}{4}}w)_3=0$.
	Applying \eqref{eq:H coercivity} to \eqref{eq:cutoff w equ}, we obtain
	\begin{equation*}
		\|\chi_{\frac{r_0}{4}}w\|_{\dot H^2_3}
		\lesssim \|H_{\nu_N}(\chi_{\frac{r_0}{4}}w)\|_{L^2_3}
		\lesssim\|\calT(u)\|_{L^2_3}+\lambda_{N-1}^{-1}\delta^{-1/2}
		+\delta^{1/2}(C_{r_0}+\|\chi_{\frac{r_0}{4}}w\|_{\dot H^2_3}).
	\end{equation*}
	By \eqref{eq:delta bound}, we can absorb $\delta^{1/2}\|\chi_{\frac{r_0}{4}}w\|_{\dot H^2_3}$ into the left-hand side. If $r_0<\infty$, then $\lambda_{N-1}\to0$, and we get $C_{r_0}\delta^{1/2}\leq\lambda_{N-1}^{-1}\delta^{-1/2}$.
	When $r_0=\infty$, this term is zero. Hence, we conclude \eqref{eq:H2 esti for three dim},
	\begin{equation*}
		\|\chi_{\frac{r_0}{4}}w\|_{\dot H^2_3}
		\lesssim\|\calT(u)\|_{L^2_3}+\lambda_{N-1}^{-1}\delta^{-1/2}. \qedhere
	\end{equation*}
\end{proof}

Finally, we finish the proof of Lemma~\ref{lem:nonlinear esti}.
\begin{proof}[Proof of Lemma~\ref{lem:nonlinear esti}]
	For $D\in\{1,\frac32\}$, \eqref{eq:Hardy Sobolev} and \eqref{eq:delta} imply
	\begin{equation*}
		\|\chi_{\delta_0\lambda_{N-1}}g\|_{\dot H^1_d}
		\lesssim\|\chi_{2r_0}g\|_{\dot H^1_d}\leq\delta^{1/2}.
	\end{equation*}
	For $D=\frac12$, the corresponding estimate is \eqref{eq:H1 esti for three dim}. Applying \eqref{eq:weighted interpolation} to $\chi_{\delta_0\lambda_{N-1}}g$ when $D\in\{1,\frac32\}$ and to $\chi_{\frac{r_0}{4}}w$ when $D=\frac12$, and then using \eqref{eq:H2 esti for D one three halves} and \eqref{eq:H2 esti for three dim}, respectively, we obtain
	\begin{align}
		{\textstyle\int_0^{\delta_0\lambda_{N-1}}}g^2r^{D-1}dr
		&\lesssim\delta^{(2-D)/2}\|\calT(u)\|_{L^2_d}^D
		+\delta^{(2-D)/2}\lambda_{N-1}^{-D},
		&&D\in\{1,\tfrac32\},\label{eq:interior weighted esti high d}\\
		{\textstyle\int_0^{r_0/4}}w^2r^{D-1}dr
		&\lesssim\delta^{(2-D)/2}\|\calT(u)\|_{L^2_3}^D
		+\delta^{1/2}\lambda_{N-1}^{-D},
		&&D=\tfrac12. \label{eq:interior weighted esti 3 d}
	\end{align}
	For $D\in\{1,\frac32\}$, the remaining parts satisfy
	\begin{equation}\label{eq:exterior weighted esti}
		{\textstyle\int_{\delta_0\lambda_{N-1}}^{r_0}}g^2r^{D-1}dr
		\lesssim\lambda_{N-1}^{-D}\|r^{-1}g\|_{L^2_d(0,r_0)}^2
		\lesssim\delta^{1/2}\lambda_{N-1}^{-D},
		\qquad D\in\{1,\tfrac32\}.
	\end{equation}
	Thus, \eqref{eq:interior weighted esti high d} and \eqref{eq:exterior weighted esti} prove \eqref{eq:weighted esti} for $D\in\{1,\frac32\}$.
	
	When $D=\frac12$ and $r_0<\infty$, the remaining parts become
	\begin{equation*}
		{\textstyle\int_{r_0/4}^{r_0}}w^2r^{D-1}dr
		\lesssim C_{r_0}\|\chi_{2r_0}w\|_{\dot H^1_3}^2
		\lesssim C_{r_0}\delta
		\lesssim\delta^{1/2}\lambda_{N-1}^{-D}.
	\end{equation*}
	There is no remaining part when $D=\frac12$ and $r_0=\infty$. Hence
	\begin{equation}\label{eq:w weighted esti}
		{\textstyle\int_0^{r_0}}w^2r^{D-1}dr
		\lesssim\delta^{(2-D)/2}\|\calT(u)\|_{L^2_3}^D
		+\delta^{1/2}\lambda_{N-1}^{-D},
		\qquad D=\frac12.
	\end{equation}
	We now recover $g$ from \eqref{eq:modified error} and \eqref{eq:modified profile},
	\begin{equation*}
		g=w+\td U,\qquad
		\td U=\iota_N(W_{\nu_N}-W_{\lambda_N})
		-({\textstyle\sum_{j=1}^{N-1}}\iota_j\lambda_j^{-D})\chi_{4\delta\lambda_{N-1}}.
	\end{equation*}
	The scale invariance of $\Lambda W$ and \eqref{eq:nuN correction} imply
	\begin{align*}
		{\textstyle\int_0^\infty}|W_{\nu_N}-W_{\lambda_N}|^2r^{D-1}dr
		&\lesssim|\log(\nu_N/\lambda_N)|^2\lambda_N^{-D}
		\lesssim\mu_N^D\lambda_{N-1}^{-D},\\
		{\textstyle\int_0^\infty}|({\textstyle\sum_{j=1}^{N-1}}\iota_j\lambda_j^{-D})
		\chi_{4\delta\lambda_{N-1}}|^2r^{D-1}dr
		&\lesssim\delta^{1/2}\lambda_{N-1}^{-D}.
	\end{align*}
	Therefore
	\begin{equation}\label{eq:modified corrector weight esti}
		{\textstyle\int_0^\infty}(\td U)^2r^{D-1}dr
		\lesssim\delta^{1/2}\lambda_{N-1}^{-D}.
	\end{equation}
	Thus, the estimate \eqref{eq:weighted esti} follows from $g=w+\td U$, \eqref{eq:w weighted esti}, and \eqref{eq:modified corrector weight esti}.
\end{proof}

\bibliographystyle{abbrv}
\bibliography{reference}

\end{document}